\documentclass[preprint]{imsart}

\RequirePackage[OT1]{fontenc}
\RequirePackage{amsthm,amsmath}
\RequirePackage[numbers]{natbib}
\RequirePackage[colorlinks,citecolor=blue,urlcolor=blue]{hyperref}
\usepackage{amssymb,amsfonts,amsmath,amsthm,amscd}
\usepackage{PSFigure,psfrag}
\usepackage{graphicx}

\def\IZ{{\mathbb Z}}
\def\IR{{\mathbb R}}
\def\IP{{\mathbb P}}

\def\IL{{\mathbb L}}
\def\IE{{\mathbb E}}
\def\cA{{\cal A}}
\def\cB{{\cal B}}
\def\cC{{\cal C}}
\def\cD{{\cal D}}
\def\cE{{\cal E}}

\def\cH{{\cal H}}
\def\cI{{\cal I}}
\def\cK{{\cal K}}

\def\cS{{\cal S}}
\def\cU{{\cal U}}
\def\cV{{\cal V}}
\def\cW{{\cal W}}

\def\wt{\widetilde}
\def\wh{\widehat}
\def\r{\rightarrow}
\def\ov{\overline}
\def\dis{\displaystyle}
\def\f{\footnotesize}
\def\no{\normalsize}
\def\ve{\varepsilon}

\startlocaldefs
\numberwithin{equation}{section}
\theoremstyle{plain}
\newtheorem{thm}{Theorem}[section]
\newtheorem{rmk}{Remark}[section]
\newtheorem{lem}{Lemma}[section]
\endlocaldefs

\begin{document}

\begin{frontmatter}
\title{On macroscopic holes in some supercritical strongly dependent percolation models} 

\runtitle{On macroscopic holes}


\begin{aug}
\author{\fnms{Alain-Sol} \snm{Sznitman}
\ead[label=e1]
{sznitman@math.ethz.ch}}


\affiliation{ETH Z\"urich
}

\address{Departement Mathematik\\
ETH Z\"urich\\
R\"amistrasse 101\\
CH-8092 Z\"urich
\\
}

\end{aug}

\begin{abstract}
We consider $\IZ^d$, $d \ge 3$. We investigate the vacant set $\cV^u$ of random interlacements in the strongly percolative regime, the vacant set $\cV$ of the simple random walk, and the excursion set $E^{\ge \alpha}$ of the Gaussian free field in the strongly percolative regime. We consider the large deviation probability that the adequately thickened component of the boundary of a large box centered at the origin in the respective vacant sets or excursion set leaves in the box a macroscopic volume in its complement. We derive asymptotic upper and lower exponential bounds for theses large deviation probabilities. We also derive geometric information on the shape of the left-out volume. It is plausible, but open at the moment, that certain critical levels coincide, both in the case of random interlacements and of the Gaussian free field. If this holds true, the asymptotic upper and lower bounds that we obtain are matching in principal order for all three models, and the macroscopic holes are nearly spherical. We heavily rely on the recent work \cite{NitzSzni} by Maximilian Nitzschner and the author for the coarse graining procedure, which we employ in the derivation of the upper bounds.
\end{abstract}



\end{frontmatter}

\section{Introduction}

In this article we consider random interlacements, the simple random walk, and the Gaussian free field on $\IZ^d$, $d \ge 3$. We are interested in the vacant set $\cV^u$ of random interlacements in the strongly percolative regime, in the vacant set $\cV$ of the simple random walk, and in the excursion set $E^{\ge \alpha}$ of the Gaussian free field in the strongly percolative regime. We consider a large box centered at the origin of side-length $2N$, and investigate the asymptotics of the large deviation probabilities that the adequately thickened components of the boundary of the box in the respective vacant sets leave out in the box a macroscopic volume in their complement. We derive exponential upper and lower bounds on these probabilities, which involve certain critical levels for the random interlacements and for the Gaussian free field. It is plausible, but open at the moment, that these critical levels actually coincide for random interlacements, and also in the case of the Gaussian free field. If these equalities hold, the upper and lower exponential bounds that we derive here, match in principal order, and the above mentioned sets left out in the box are close to spherical holes. Several of the results presented in this article rely on the recent solidification estimates of porous interfaces and the coarse graining procedure developed in \cite{NitzSzni}. The questions that we investigate share a similar flavor to phase separation and the emergence of a macroscopic Wulff shape for Bernoulli percolation or for the Ising model, see \cite{Cerf00}, \cite{Bodi99}. But there are notable differences. In particular, the exponential costs for the long range models discussed here involve capacity and not perimeter.

We will now describe our results in more detail. We begin with the case of random interlacements. We consider $\IZ^d$, $d \ge 3$, and for $u > 0$ denote by $\cI^u$ and $\cV^u = \IZ^d \backslash \cI^u$ the interlacement at level $u$ and its corresponding vacant set. We denote by $\IP$ the probability governing the random interlacements and refer to \cite{CernTeix12}, \cite{DrewRathSapo14c}, for background material over the model. Given an integer $N \ge 1$, we write
\begin{equation}\label{0.1}
B(0,N) = [-N,N]^d \cap \IZ^d \; \;\mbox{and}\; \;S_N = \{x \in \IZ^d; |x|_\infty = N\},
\end{equation}
for the closed $|\cdot |_\infty$-ball in $\IZ^d$ with center $0$ and radius $N$, and for its inner boundary (accordingly, we denote by $B(x,r)$ the closed $|\cdot |_\infty$-ball in $\IZ^d$ with center $x \in \IZ^d$ and radius $r \ge 0$). We then consider
\begin{equation}\label{0.2}
\begin{split}
\cC^u_N = & \;\mbox{the connected component of $S_N$ in $\cV^u \cup S_N$}
\\
&\mbox{(so, by convention we have $S_N \subseteq \cC^u_N)$}
\\[1ex]
 =& \;\mbox{the collection of sites in $\IZ^d$ that either belong to $S_N$ or are}
\\[-0.5ex]
& \;\mbox{connected by a finite nearest-neighbor path with end point}
\\[-0.5ex]
&\; \mbox{in $S_N$, such that each site of the path, except maybe for the}
\\[-0.5ex]
& \;\mbox{last one, belongs to $\cV^u$}. 
\end{split}
\end{equation}
Further, we consider a sequence of non-negative integers $\wt{L}_0(N)$, such that
\begin{equation}\label{0.3}
\wt{L}_0(N) = o(N), \;\mbox{as} \; N \r \infty.
\end{equation}
We use $\wt{L}_0(N)$ to thicken $\cC^u_N$, and define
\begin{equation}\label{0.4}
\mbox{$\wt{\cC}^u_N =$ the $\wt{L}_0$-neighborhood of $\cC^u_N$ in $|\cdot |_\infty$-distance},
\end{equation}
as well as the respective complements of $\cC^u_N$ and $\wt{\cC}^u_N$ in $B(0,N)$:
\begin{align}
\cW^u_N & = B(0,N) \;\backslash\; \cC^u_N, \;\;\mbox{and} \label{0.5}
\\[1ex]
\wt{\cW}^u_N & = B(0,N) \;\backslash \;\wt{\cC}^u_N \label{0.6}
\end{align}
(incidentally, note that $\wt{\cW}_N^u$ decreases when $\wt{L}_0(N)$ is replaced by a bigger sequence). Informally, $\wt{\cW}^u_N$ corresponds to the ``hole left out in $B(0,N)$ by the thickening $\wt{\cC}^u_N$ of the component of $S_N$ in $\cV^u \cup S_N$''. As mentioned above, we are interested in the event that $\wt{\cW}^u_N$ has a macroscopic volume of order $N^d$. The strongly non-percolative regime of $\cV^u$ corresponds to $u > u_{**}$, and the strongly percolative regime corresponds to $0 < u < \ov{u}$, with $u_{**}$ and $\ov{u}$ as in (0.2) and (2.3) of \cite{Szni17}. One has $0 < \ov{u} \le u_* \le u_{**} < \infty$ with $u_*$ the critical level for the percolation of the vacant set $\cV^u$ (the positivity of $\ov{u}$ for all $d\ge3$ is due to \cite{DrewRathSapo14a}). It is plausible, but open at the moment, that $\ov{u} = u_* = u_{**}$ (and some possible progress towards proving $u_* = u_{**}$ may come from \cite{DumiRaouTass17}). In the strongly non-percolative regime corresponding to $u > u_{**}$, the probability that there is a path in $\cV^u$ between $0$ and $x$ decays exponentially in $|x|_\infty$, when $d \ge 4$, with a logarithmic correction, when $d = 3$, see Theorem 3.1 of \cite{PopoTeix15}. From this fact, it readily follows that
\begin{equation}\label{0.7}
\begin{array}{l}
\mbox{when $u > u_{**}$, for any $\wt{L}_0(N)$ as in (\ref{0.3}), $|\wt{\cW}^u_N| \,/ \,|B(0,N)| \underset{N}{\longrightarrow} 1$}
\\[-0.5ex]
\mbox{in $\IP$-probability,} 
\end{array}
\end{equation}
where for $A$ finite subset of $\IZ^d$ we let $|A|$ denote the number of elements of $A$.

We will see in Section 3 that $|\cW^u_N|$ behaves differently in the strongly percolative regime $0 < u < \ov{u}$. In particular, we show in Theorem \ref{theo2.1} that
\begin{equation}\label{0.8}
\left\{ \begin{array}{l}
\mbox{there exists $\wt{L}_0(N)$ satisfying (\ref{0.3}) such that for any}
\\
\mbox{$0 < u < \ov{u}$ and $\nu > 0$,}
\\
\mbox{$\limsup\limits_N \;\frac{1}{N^{d-2}} \log \IP[| \wt{\cW}_N^u| \ge \nu N^d] \le - \frac{1}{d} \;(\sqrt{\ov{u}} - \sqrt{u})^2 {\rm cap}(B_\nu)$},
\end{array}\right.
\end{equation}
where $B_\nu$ stands for the closed Euclidean ball centered at the origin in $\IR^d$ with volume $\nu$ and ${\rm cap}(B_\nu)$ for its Brownian capacity (see for instance \cite{PortSton78}, p.~57-58). Incidentally, some thickening $\wt{L}_0(N)$ of the component $\cC^u_N$ is required for the aymptotic upper bound (\ref{0.8}) to hold true, see Remark \ref{rem2.5} 4). How small $\wt{L}_0(N)$ can be chosen is a presently open question.

We now turn to the geometric controls. We consider $F^u_N$ the $\IR^d$-filling of $\frac{1}{N} \;\wt{\cW}_N^u$ $(\subseteq \frac{1}{N} \;\IZ^d)$, namely:
\begin{equation}\label{0.9}
F^u_N = \big\{z \in \IR^d; d_\infty \big(z, \mbox{\normalsize $\frac{1}{N}$} \;\wt{\cW}^u_N\big) \le \mbox{\normalsize $\frac{1}{N}$}\big\},
\end{equation}
where for $z \in \IR^d$, $A \subseteq \IR^d$, $d_\infty (z,A) =  \inf\{|z-a|_\infty$; $a \in A\}$ denotes the sup-norm distance of $z$ to $A$. We let $\delta(F^u_N, B_\nu)$ stand for the minimal volume of the symmetric difference of $F^u_N$ with a translate of $B_\nu$ (see also (\ref{1.11})):
\begin{equation}\label{0.10}
\delta(F^u_N, B_\nu) = \min\limits_{z \in \IR^d} \;|F^u_N \, \Delta (B_\nu + z)|
\end{equation}
(with $\Delta$ the symmetric difference and $|A|$ the volume of $A$ Borel subset of $\IR^d$).

We further prove in Theorem \ref{theo2.1} that with $\wt{L}_0(N)$ as in (\ref{0.8}),
\begin{equation}\label{0.11}
\begin{array}{l}
 \mbox{for any $0 < u < \ov{u}$, $\nu > 0$, and $\mu > 0$, one has} 
 \\[1ex]
\limsup\limits_N \;\frac{1}{N^{d-2}} \log \IP[|\wt{\cW}^u_N|  \ge \nu N^d,\, \delta(F^u_N, B_\nu) \ge \mu] \le
\\[-0.7ex]
\hspace{4.3cm} - \frac{1}{d} (\sqrt{\ov{u}} - \sqrt{u})^2 \big({\rm cap}(B_\nu) + c_1 (\nu, \mu)\big)
\end{array}
\end{equation}
with $c_1(\nu,\mu)$ a positive constant solely depending on $d,\nu$ and $\mu$.

We also derive lower bounds in Section 3. If we let
\begin{equation}\label{0.12}
\mbox{$\omega_d =$ the volume of a Euclidean ball in $\IR^d$ of unit radius,}
\end{equation}
we show in Theorem \ref{theo2.4} that
\begin{equation}\label{0.13}
\left\{\begin{array}{l}
 \mbox{for any $\wt{L}_0(N)$ as in (\ref{0.3}), $0 < u < u_{**}$ and $0  < \nu < \omega_d$},
 \\
\liminf\limits_N \, \frac{1}{N^{d-2}} \log \IP[|\wt{\cW}^u_N| \ge \nu N^d] \ge - \frac{1}{d} (\sqrt{u}_{**} - \sqrt{u})^2 {\rm cap}(B_\nu).
\end{array}\right.
\end{equation}
The restriction $\nu < \omega_d$ in the lower bound implies that $B_\nu$ is contained in $(-1,1)^d$. This feature enables us to avoid boundary effects: it ensures that $B_\nu$ is a subset of $(-1,1)^d$, which has minimal capacity among all subsets of $(-1,1)^d$ with volume $\nu$ (see also the end of Section 2). Note that unlike the case of the upper bound (\ref{0.8}), no thickening of the component $\cC^u_N$ is needed for the lower bound  (\ref{0.13}) (i.e. one may choose $\wt{L}_0(N) = 0$). Incidentally, when $0 < u < u_{**}$, $0  < \nu < \omega_d$, and $0  < \mu < \omega_d - \nu$, the asymptotic lower bound (\ref{0.13}) applied with $\nu+\mu$ in place of $\nu$, and the observation that $\delta(F^u_N, B_\nu) \ge \mu$ when $|\wt{\cW}^u_N| \ge (\nu + \mu) N^d$, readily provides a similar looking asymptotic lower bound for the probability in (\ref{0.11}) with $u_{**}$ in place of $\ov{u}$, and $c'_1 (\nu, \mu) = {\rm cap}(B_{\nu + \mu}) - {\rm cap}(B_\nu)$ in place of $c_1 (\nu, \mu)$.

If the critical levels $\ov{u} \le u_* \le u_{**}$ coincide, then Theorems \ref{theo2.1} and \ref{theo2.4} show that with $\wt{L}_0(N)$ as in (\ref{0.8})
\begin{equation}\label{0.14}
\begin{array}{l}
\lim\limits_N \;\mbox{\no $\frac{1}{N^{d-2}}$} \log \IP[|\wt{\cW}^u_N| \ge \nu N^d] = - \mbox{\no $\frac{1}{d}$} (\sqrt{u}_* - \sqrt{u})^2 {\rm cap} (B_\nu), 
\\[1ex]
\mbox{for $0 < u < u_*$ and $0 < \nu < \omega_d$,}
\end{array}
\end{equation}
and that conditionally on $\{|\wt{\cW}^u_N| \ge \nu N^d\}$, the set $F^u_N$ is close to a translate of $B_\nu$:
\begin{equation}\label{0.15}
\begin{array}{l}
\lim\limits_N \IE\big[\delta (F^u_N, B_\nu) \,\big| \, |\wt{\cW}^u_N| \ge \nu N^d\big] = 0, 
\\
\mbox{for $0 < u < u_*$ and $0 < \nu < \omega_d$}.
\end{array}
\end{equation}
What happens inside the hole $\wt{\cW}_N^u$ left in $B(0,N)$ by $\wt{\cC}^u_N$ is however unclear, see Remark \ref{rem2.5} 3). Still assuming the equality of $\ov{u} \le u_* \le u_{**}$, the asymptotics (\ref{0.14}) and (\ref{0.15}) should also be contrasted with (\ref{0.7}) now governing the case $u > u_*$.

Let us now turn to the case of the simple random walk, which at a heuristic level corresponds to setting $u = 0$ for the random interlacements. We denote by $(X_n)_{n \ge 0}$ the canonical simple random walk on $\IZ^d$ (we recall that $d \ge 3)$, by $P_x$ its canonical law when starting from $x$ in $\IZ^d$, and by $E_x$ the corresponding expectation. We write $\cI = \{X_n; n \ge 0\} \subseteq \IZ^d$ for the set of points visited by the walk, and $\cV = \IZ^d \backslash \cI$ for the corresponding vacant set. In analogy with (\ref{0.2}), we introduce for $N \ge 1$
\begin{equation}\label{0.16}
\begin{array}{l}
\mbox{$\cC_N =$ the connected component of $S_N$ in $\cV \cup S_N$}
\\
\mbox{(so, by convention $S_N \subseteq \cC_N$)}.
\end{array}
\end{equation}
For a sequence $\wt{L}_0(N)$ as in (\ref{0.3}), we then define
\begin{equation}\label{0.17}
\mbox{$\wt{\cC}_N =$ the $\wt{L}_0$-neighborhood of $\cC_N$ in $| \cdot |_\infty$-distance,}
\end{equation}
as well as the respective complements of $\cC_N$ and $\wt{\cC}_N$ in $B(0,N)$:
\begin{align}
\cW_N = &B(0,N) \; \backslash \; \cC_N, \label{0.18}
\\[1ex]
\wt{\cW}_N = &B(0,N) \; \backslash \; \wt{\cC}_N, \label{0.19}
\end{align}
In Theorem \ref{theo3.1}, we show that
\begin{equation}\label{0.20}
\left\{ \begin{array}{l}
\mbox{with $\wt{L}_0(N)$ as in (\ref{0.8}), for any $\nu \ge 0$},
\\[1ex]
\limsup\limits_N \;\frac{1}{N^{d-2}} \log P_0 [| \wt{\cW}_N| \ge \nu N^d] \le -  \frac{1}{d} \;\ov{u} \;{\rm cap} (B_\nu).
\end{array}\right.
\end{equation}
Introducing the $\IR^d$-filling of $\frac{1}{n} \,\wt{\cW}_N (\subseteq \frac{1}{N} \,\IZ^d)$, namely
\begin{equation}\label{0.21}
F_N = \big\{z \in \IR^d; d_\infty\big(z, \mbox{\no $\frac{1}{N}$} \;\wt{\cW}_N\big) \le \mbox{\no $\frac{1}{N}$}\big\},
\end{equation}
one has with similar notation as in (\ref{0.10}) (see also (\ref{1.11})):
\begin{equation}\label{0.22}
\begin{split}
\mbox{for any $\nu > 0$ and $\mu > 0$},\
\\[1ex]
\limsup\limits_N \;\mbox{\no $\frac{1}{N^{d-2}}$} \log P_0 [| \wt{\cW}_N|& \ge \nu N^d,\,\delta(F_N,B_\nu) \ge \mu]  
\\[-1ex]
&\le - \mbox{\no $\frac{1}{d}$} \;\ov{u} \big({\rm cap}(B_\nu) + c_2 (\nu, \mu)\big)
\end{split}
\end{equation}
with $c_2(\nu,\mu)$ a positive constant solely depending on $d,\nu,\mu$.

We derive a lower bound in Theorem \ref{theo3.2} and show that
\begin{equation}\label{0.23}
\left\{ \begin{array}{l}
\mbox{for any $\wt{L}_0(N)$ as in (\ref{0.3}) and $0 < \nu < \omega_d$},
\\[1ex]
\liminf\limits_N \;\mbox{\no $\frac{1}{N^{d-2}}$} \log P_0 [| \wt{\cW}_N| \ge \nu N^d] \ge - \mbox{\no $\frac{1}{d}$} \;u_{**} \,{\rm cap}(B_\nu).
\end{array}\right.
\end{equation}
Along the same lines as explained above (\ref{0.14}), when $0  < \nu < \omega_d$, and $0  < \mu < \omega_d - \nu$, the asymptotic lower bound (\ref{0.23}) can be used to produce an asymptotic lower bound for the probability in (\ref{0.22}) with $u_{**}$ in place of $\ov{u}$, and $c'_1 (\nu, \mu) = {\rm cap}(B_{\nu + \mu}) - {\rm cap}(B_\nu)$ in place of $c_2 (\nu, \mu)$.

Again, if the equalities $\ov{u} = u_* = u_{**}$ hold, then with $\wt{L}_0(N)$ as in (\ref{0.8}),
\begin{equation}\label{0.24}
\lim\limits_N \; \mbox{\no $\frac{1}{N^{d-2}}$} \log P_0 [| \wt{\cW}_N| \ge \nu N^d]  = -\mbox{\no $\frac{u_*}{d}$} \; {\rm cap} (B_\nu), \; \mbox{for $0 < \nu < \omega_d$},
\end{equation}
and conditionally on $\{|\wt{\cW}_N| \ge \nu N^d\}$, the set $F_N$ is close to a translate of $B_\nu$:
\begin{equation}\label{0.25}
\lim\limits_N E_0\big[\delta(F_N,B_\nu) \,\big| \, |\wt{\cW}_N| \ge \nu N^d\big] = 0, \;\mbox{for $0 < \nu < \omega_d$}.
\end{equation}
Incidentally, what happens inside the hole $\wt{\cW}_N$ left in $B(0,N)$ by $\wt{\cC}_N$ is unclear, see Remark \ref{rem3.3} 2).

We then turn to the results corresponding to the Gaussian free field. In place of $\cV^u$ or $\cV$, we now consider the excursion sets $E^{\ge \alpha} = \{x \in \IZ^d$; $\varphi_x \ge \alpha\}$, where $(\varphi_x)_{x \in \IZ^d}$ stands for the canonical Gaussian free field on $\IZ^d$, $d \ge 3$. We let $\IP^G$ stand for its canonical law and $\IE^G$ for the corresponding expectation. There are now critical levels $-\infty < \ov{h} \le h_* \le h_{**} < \infty$, see (\ref{0.4}), (\ref{0.5}), and (5.3) of \cite{Szni15}, so that $E^{\ge \alpha}$ is in the strongly non-percolative regime for $h > h_{**}$, in the strongly percolative regime for $h < \ov{h}$, and $h_*$ denotes the threshold for the percolation of $E^{\ge \alpha}$. It has recently been proved that $h_* > 0$, cf.~\cite{DrewPrevRodr}, this was previously only known for large $d$, see \cite{RodrSzni13}, \cite{DrewRodr15}. It is plausible, but open at the moment, that actually $\ov{h} = h_* = h_{**}$ (possibly some progress in proving $h_* = h_{**}$ may come from \cite{DumiRaouTass17}). The investigation of level-set percolation of the Gaussian free field was launched in \cite{LeboSale86}, \cite{BricLeboMaes87}, see also \cite{MolcStep83}. There are also deep links between random interlacements and the level sets of the Gaussian free field, see \cite{Lupu16} and Section 2 of \cite{Szni16}. Further, the results of \cite{Szni15} on disconnection have recently been sharpened in \cite{Nitz} and \cite{ChiaNitz}  with the help of the methods of \cite{NitzSzni}.

In analogy with what we did for random interlacements and for the simple random walk, we introduce for $\alpha \in \IR$ and $N \ge 1$
\begin{equation}\label{0.26}
\begin{array}{l}
\mbox{$\cC^{\ge \alpha}_N =$ the connected component in $E^{\ge \alpha} \cup S_N$ of $S_N$}
\\
\mbox{(so, by convention $S_N \subseteq \cC^{\ge \alpha}_N$)}.
\end{array}
\end{equation}
For $\wt{L}_0(N)$, a sequence as in (\ref{0.3}), we then define
\begin{equation}\label{0.27}
\mbox{$\wt{\cC}^{\ge \alpha}_N =$ the $\wt{L}_0$-neighborhood of $\cC^{\ge \alpha}_N$ in the $|\cdot |_\infty$-distance},
\end{equation}
as well as the respective complements of $\cC^{\ge \alpha}_N$ and $\wt{\cC}^{\ge \alpha}_N$ in $B(0,N)$:
\begin{align}
\cW^{\ge \alpha}_N = & B(0,N) \;\backslash \;\cC^{\ge \alpha}_N, \label{0.28}
\\[1ex]
\wt{\cW}^{\ge \alpha}_N = & B(0,N)\; \backslash \; \wt{\cC}_N^{\ge \alpha}. \label{0.29}
\end{align}
When $\alpha > h_{**}$, the excursion set $E^{\ge \alpha}$ is in a strongly non-percolative regime, and the probability that there is a path in $E^{\ge \alpha}$ between $0$ and $x$ decays exponentially in $|x|_\infty$, when $d \ge 4$, with a logarithmic correction, when $d = 3$, cf.~\cite{PopoRath15}. As a result, for $\alpha > h_{**}$, $\wt{\cW}_N^{\ge \alpha}$ has nearly full volume in $B(0,N)$ and analogously to (\ref{0.7})
\begin{equation}\label{0.30}
\begin{array}{l}
\mbox{when $\alpha > h_{**}$, for any $\wt{L}_0(N)$ as in (\ref{0.3})}, 
\\[0.5ex]
\mbox{$|\wt{\cW}_N^{\ge \alpha}| / | B(0,N)| \underset{N}{\longrightarrow} 1$ in $\IP^G$-probability}.
\end{array}
\end{equation}
As we show in Theorem \ref{theo4.1}, the behavior is different in the strongly percolative regime $\alpha < \ov{h}$:
\begin{equation}\label{0.31}
\left\{ \begin{array}{l}
\mbox{there exists $\wt{L}_0(N)$ satisfying (\ref{0.3}) such that}
\\
\mbox{for any $\alpha < \ov{h}$ and $\nu > 0$,}
\\
\limsup\limits_N \;\mbox{\no $\frac{1}{N^{d-2}}$} \;\log \IP^G[ | \wt{\cW}^{\ge \alpha}_N | \ge \nu N^d] \le - \mbox{\no $\frac{1}{2d}$} \;(\ov{h} - \alpha)^2 \,{\rm cap} (B_\nu)
\end{array}\right.
\end{equation}
(with $B_\nu$ as in (\ref{0.8})).

In addition, if $F_N^{\ge \alpha}$ denotes the $\IR^d$-filling of $\frac{1}{N} \; \wt{\cW}_N^{\ge \alpha}$ $(\subseteq \frac{1}{N} \;\IZ^d)$:
\begin{equation}\label{0.32}
F^{\ge \alpha}_N = \big\{z \in \IR^d; \; d_\infty \big(z,\mbox{\no $\frac{1}{N}$} \;\wt{\cW}^{\ge \alpha}_N\big) \le \mbox{\no $\frac{1}{N}$}\big\},
\end{equation}

with similar notation as in (\ref{0.10}) (see also (\ref{1.11})), we show that with $\wt{L}_0(N)$ as in (\ref{0.31}), we have
\begin{equation}\label{0.33}
\begin{array}{l}
\mbox{for any $\alpha < \ov{h}$, $\nu > 0$ and $\mu \ge 0$,}
\\[1ex]
\limsup\limits_N \;\mbox{\no $\frac{1}{N^{d-2}}$} \;\log \IP^G[ | \wt{\cW}^{\ge \alpha}_N|  \ge \nu N^d, \delta(F^{\ge \alpha}_N, B_\nu) \ge \mu]  
\\
\le - \mbox{\no $\frac{1}{2d}$} \;(\ov{h} - \alpha)^2 \big({\rm cap}(B_\nu) + c_1 (\nu, \mu)\big)
\end{array}
\end{equation}
(with $c_1(\nu,\mu)$ as in (\ref{0.11})).

We also derive an asymptotic lower bound in Theorem \ref{theo4.2}:
\begin{equation}\label{0.34}
\left\{ \begin{array}{l}
\mbox{for any $\wt{L}_0(N)$ as in (\ref{0.3}), $\alpha < h_{**}$, and $0 < \nu < \omega_d$, one has}
\\[1ex]
\liminf\limits_N \;\mbox{\no $\frac{1}{N^{d-2}}$} \;\log \IP^G[ | \wt{\cW}^{\ge \alpha}_N | \ge \nu N^d] \ge - \mbox{\no $\frac{1}{2d}$} \;(h_{**} - \alpha)^2 {\rm cap}(B_\nu)
\end{array}\right.
\end{equation}
(with $\omega_d$ as in (\ref{0.12})). By a similar argument as explained above (\ref{0.14}), when $\alpha < h_{**}$, $0  < \nu < \omega_d$, and $0  < \mu < \omega_d - \nu$, the asymptotic lower bound (\ref{0.34}) can be used to produce an asymptotic lower bound for the probability in (\ref{0.33}) with $h_{**}$ in place of $\ov{h}$, and $c'_1 (\nu, \mu) = {\rm cap}(B_{\nu + \mu}) - {\rm cap}(B_\nu)$ in place of $c_1 (\nu, \mu)$.

Again, if the equalities $\ov{h} = h_* = h_{**}$ hold, then Theorem \ref{theo4.1} and \ref{theo4.2} show that with $\wt{L}_0(N)$ as in Theorem \ref{theo4.1}, one has
\begin{equation}\label{0.35}
\begin{array}{l}
\lim\limits_N \; \mbox{\no $\frac{1}{N^{d-2}}$} \;\log \IP^G[ | \wt{\cW}^{\ge \alpha}_N | \ge \nu N^d] = - \mbox{\no $\frac{1}{2d}$}\; (h_* - \alpha)^2\, {\rm cap}(B_\nu)
\\
\mbox{for $\alpha < h_*$ and $0 < \nu < \omega_d$},
\end{array}
\end{equation}
and conditionally on $\{ | \wt{\cW}^{\ge \alpha}_N | \ge \nu N^d\}$, the set $F_N^{\ge \alpha}$ is close to a translate of $B_\nu$:
\begin{equation}\label{0.36}
\begin{array}{l}
\lim\limits_N \; \IE^G[\delta(F^{\ge \alpha}_N, B_\nu)  \, \big|\, |\wt{\cW}^{\ge \alpha}_N | \ge  \nu N^d] = 0, 
\\
\mbox{for $\alpha < h_*$ and $0 < \nu < \omega_d$},
\end{array}
\end{equation}
and (\ref{0.35}), (\ref{0.36}) should also be contrasted with (\ref{0.30}) now governing the case $\alpha > h_*$ (when $\ov{h} \le h_* \le h_{**}$ coincide). Still, what happens inside the hole $\wt{\cW}^{\ge \alpha}_N$ left in $B(0,N)$ by $\wt{\cC}^{\ge \alpha}_N$ is unclear, see Remark \ref{rem4.3} 2).

Let us say a few words about proofs. The most challenging part concerns the derivation of the upper bounds in the case of random interlacements in Theorem \ref{theo2.1}, see also (\ref{0.8}), (\ref{0.11}), and in the case of the Gaussian free field in Theorem \ref{theo4.1}, see also (\ref{0.31}), (\ref{0.33}). We heavily rely on the type of coarse graining procedure of \cite{NitzSzni}, Section 4, leading to the construction of suitable porous interfaces, and on the capacity lower bounds from \cite{NitzSzni} that are recalled in Section 1. Let us stress the following feature. Although the optimal shape $B_\nu$ ``governing the problem'' is convex, the coarse graining procedure must rule out non-convex shapes and corresponding porous interfaces and establish that they are more costly. The quantitative geometric controls in (\ref{0.11}), (\ref{0.22}), (\ref{0.33}) rely on the bounds of the Fraenkel asymmetry in terms of the capacity excess due to \cite{FuscMaggPrat09} that are recalled in Section 1. They lead to a certain coercivity property of the capacity stated in Lemma \ref{lem2.3}. The lower bounds for their part come as rather direct applications of the results of \cite{LiSzni14} in the case of random interlacements, of \cite{Li17} in the case of the simple random walk, and of \cite{Szni15} in the case of the Gaussian free field.

Let us now describe the organization of the article. In Section 2 we recall the asymptotic capacity lower bounds from \cite{NitzSzni}, the change of probability method, and the bounds on the Fraenkel asymmetry in terms of the capacity excess due to \cite{FuscMaggPrat09}. In Section 3 we treat the case of random interlacements. The upper bounds appear in Theorem \ref{theo2.1} and the lower bounds in Theorem \ref{theo2.4}. Section 4 contains the results for the case of the simple random walk. The upper bounds appear in Theorem \ref{theo3.1} and the lower bounds in Theorem \ref{theo3.2}. Finally, Section 5 discusses the case of the Gaussian free field. The upper bounds are contained in Theorem \ref{theo4.1} and the lower bounds in Theorem \ref{theo4.2}.

Our convention concerning constants is the following. We denote by $c,c',\wt{c}$ positive constants changing from place to place that simply depend on the dimension $d$. Numbered constants $c_0,c_1,c_2$ refer to the value corresponding to their first appearance in the the text. Dependence on additional parameters appears in the notation.

\section{Some useful facts}

In this section we will recall some results that will be helpful in the next sections. In particular, we will recall a capacity lower bound from \cite{NitzSzni} that will play an important role in the derivations of the asymptotic upper bounds, both in Theorems \ref{theo2.1} and \ref{theo4.1}. We will also recall the relative entropy inequality that underpins the change of probability method in the proofs of the asymptotic lower bounds in Theorems \ref{2.4}, \ref{theo3.2}, and \ref{theo4.2}. Finally, we will recall the quantitative strengthening due to \cite{FuscMaggPrat09} of the Polya-Szeg\"o's Inequality. It provides a control on the Fraenkel asymmetry of a bounded open set in terms of its capacity excess.

Throughout the article, unless explicitly stated otherwise, we assume that $d \ge 3$. For $x$ in $\IR^d$ and $r \ge 0$, $B_{\IR^d}(x,r)$ will stand for the closed ball in $|\cdot |_\infty$-distance with center $x$ and radius $r$, whereas $B_2(x,r)$ will denote the corresponding closed Euclidean ball with center $x$ and radius $r$. For a Borel subset $A \subseteq \IR^d$, we also let $|A|$ stand for the Lebesgue measure of $A$.

We first recall the capacity lower bound for ``porous interfaces'' surrounding a compact subset $A$ in $\IR^d$. We consider a non-empty bounded Borel subset $U_0$ in $\IR^d$ and $U_1 = \IR^d \backslash U_0$ its complement, as well as $S = \partial U_0 = \partial U_1$ their boundary. We measure the local density of $U_1$ at $x$ in dyadic scales via
\begin{equation}\label{1.1}
\wh{\sigma}_\ell(x) = |B_{\IR^d}(x,2^{-\ell}) \cap U_1| \, / \, |B_{\IR^d}(x,2^{-\ell})|, \;\mbox{for $x \in \IR^d$ and $\ell \in \IZ$}.
\end{equation}
For $A$ compact subset of $\IR^d$ and $\ell_* \ge 0$, we define
\begin{equation}\label{1.2}
\begin{split}
\cU_{\ell_*,A} = & \;\mbox{the collection of bounded subsets $U_0$ of $\IR^d$ such that}
\\[-0.5ex]
& \;\mbox{$\wh{\sigma}_\ell(x) \le \mbox{\no $\frac{1}{2}$}$ for all $x \in A$ and $\ell \ge \ell_*$}
\end{split}
\end{equation}
(for instance, any $U_0$ such that $d_\infty(x,U_1) \ge 2^{-\ell_*}$ for each $x$ in $U_0$ belongs to $\cU_{\ell_*,A}$, where the notation $d_\infty(x,U_1)$ stands for the $| \cdot|_\infty$-distance from $x$ to $U_1$, see below (\ref{0.9})).

Given $U_0$ a bounded non-empty Borel subset of $\IR^d$, and $\ve > 0$, $\eta > 0$, the ``porous interfaces'' we consider correspond to
\begin{equation}\label{1.3}
\begin{split}
\cS_{U_0,\ve,\eta} = & \;\mbox{the class of compact subsets $\Sigma$ of $\IR^d$ such that for all}
\\[-0.5ex]
& \;\mbox{$z\in \partial U_0$, $W_z[$Brownian motion enters $\Sigma$ strictly before}
\\[-0.5ex]
&\; \mbox{moving at $| \cdot |_\infty$-distance $\ve] \ge \eta$}
\end{split}
\end{equation}
(where $W_z$ denotes the Wiener measure starting at $z$ and that governs the canonical Brownian motion in (\ref{1.3})).

Thus, $\ve$ controls the distance from $S = \partial U_0$ at which the porous interface $\Sigma$ is felt and $\eta$ the strength with which it is felt. We can now quote the capacity lower bound contained in (3.16) of Corollary 3.4 of \cite{NitzSzni} that will play an important role in the derivations of Theorems \ref{theo2.1} and \ref{theo4.1}. It states that for all $\eta \in (0,1)$ one has:
\begin{equation}\label{1.4}
\lim\limits_{v \r 0} \;\;\inf\limits_{\ve \le v\,2^{-\ell_*}} \;\;\inf\limits_A \;\;\inf\limits_{U_0 \in \cU_{\ell_*,A}} \;\;\inf\limits_{\Sigma \in \cS_{U_0,\ve,\eta}} {\rm cap}(\Sigma) \, / \, {\rm cap}(A) = 1
\end{equation}
where $A$ varies in the class of compact subsets of $\IR^d$ with positive capacity in the above infimum.

We next recall the classical inequality concerning the relative entropy that underpins the change of probability method. It will be used in the proofs of the lower bounds in Theorems \ref{2.4}, \ref{theo3.2}, and \ref{theo4.2}. Given a probability $\wt{P}$ absolutely continuous with respect to $P$, the relative entropy of $\wt{P}$ relative to $P$ is
\begin{equation}\label{1.5}
H(\wt{P} \,| \, P) = \wt{E}\,\Big[ \log \mbox{\no $\frac{d\wt{P}}{dP}$}\Big] = E\,\Big[ \mbox{\no $\frac{d\wt{P}}{dP}$} \log \mbox{\no $\frac{d\wt{P}}{dP}$} \Big] \in [0,\infty]
\end{equation}
(we denote by $\wt{E}$ and $E$ the respective expectations with respect to $\wt{P}$ and $P$). Then, the above mentioned inequality states that for an event $A$ with $\wt{P}(A) > 0$, one has (see \cite{DeusStro89}, p.~76)
\begin{equation}\label{1.6}
P[A] \ge \wt{P}[A] \;\exp\big\{- \mbox{\no $\frac{1}{\wt{P}[A]}$} \big(H(\wt{P}\,| \,P) + \mbox{\no $\frac{1}{e}$}\big)\big\}.
\end{equation}
Next, we turn to the quantitative version of Polya-Szeg\"o's Inequality derived in \cite{FuscMaggPrat09}. We recall that $\omega_d = |B_2(0,1)|$ stands for the volume of the Euclidean ball of unit radius, see (\ref{0.12}), and we let $\kappa_d$ stand for its capacity:
\begin{equation}\label{1.7}
\kappa_d = {\rm cap}(B_2(0,1)).
\end{equation}
Note that given a non-empty open set $U$ of finite volume, $\kappa_d (\frac{|U|}{\omega_d})^{d-2}$ is the capacity of a Euclidean ball with volume $|U|$. One defines the capacity excess of $U$ as:
\begin{equation}\label{1.8}
\eta_U = \mbox{\no $\frac{{\rm cap}(U)}{\kappa_d (\frac{|U|}{\omega_d})^{d-2}}$} - 1.
\end{equation}
The Polya-Szeg\"o Inequality states that $\eta_U \ge 0$. As shown in \cite{FuscMaggPrat09}, $\eta_U$ actually controls the Fraenkel asymmetry of $U$. We recall that for a Borel subset of $E$ of $\IR^d$, with $0 < |E| < \infty$, the Fraenkel asymmetry of $E$ is (with $\omega_d \,R^d_E = |E|$),
\begin{equation}\label{1.9}
\lambda_E = \min\limits_{a \in \IR^d} \;  \mbox{\no $\frac{|E \Delta B_2(a,R_E)|}{|E|}$} \in [0,2) \; \mbox{(with $\Delta$ the symmetric difference)}.
\end{equation}
Note that $|E| \,\lambda_E$ is the $L^1(dx)$-distance (where $dx$ stands for the Lebesgue measure) of $1_E$ to the closed subset of $L^1(dx)$ consisting of translates of $1_{B_2(0,R_E)}$ (that is, $|E| \,\lambda_E = \delta(E,B(0,R_E)$) in the notation (\ref{1.11}) below). Theorem 1.2 of \cite{FuscMaggPrat09} states that for any open set $U$ in $\IR^d$ of positive finite measure, one has (we refer to the end of the Introduction for our convention concerning constants):
\begin{equation}\label{1.10}
\eta_U \ge c_0 \,\lambda^4_U.
\end{equation}
This inequality will enter the proof of Theorems \ref{theo2.1} and \ref{theo4.1}, when we establish the statements corresponding to (\ref{0.11}) and (\ref{0.33}).

Finally, we record here the notation underlying (\ref{0.10}). We first observe that given any Borel subset $F \subseteq \IR^d$, with finite volume $|F|$, the collection of indicator functions of translates $F + z$, $z \in \IR^d$, of $F$ constitutes a closed subset of $L^1(dx)$. Thus, for $E, F$ Borel subsets of $\IR^d$ with finite volume, we set $\delta (E,F)$ to be the $L^1$-distance of $1_E$ to this set of translates, so that (with $\Delta$ standing for the symmetric difference)
\begin{equation}\label{1.11}
\delta(E,F) = \inf\limits_{z \in E} |E \Delta (F+z) \,| \,(= \delta(F,E)\;).
\end{equation}

\section{Random interlacements}

In this section we will state and prove the large deviation upper bound corresponding to (\ref{0.8}) in Theorem \ref{theo2.1} and the lower bound corresponding to (\ref{0.13}) in Theorem \ref{theo2.4}. In Theorem \ref{theo2.1}, we will also derive an additional control stated in (\ref{0.11}) on the proximity to some translate of a ball of volume $\nu$ of the $\IR^d$-filling $F^u_N$ of $\frac{1}{N} \,\wt{W}^u_N$, see (\ref{0.9}). In doing so, we will prove a certain coercivity inequality for the capacity in Lemma \ref{lem2.2} that will also be of use in Section 5. 

We keep the same notation as in the Introduction, see in particular (\ref{0.1}) to (\ref{0.6}).
 
We begin with the upper bound on the occurrence of a macroscopic volume for $\wt{\cW}^u_N$ in the strongly percolative regime $0 < u < \ov{u}$, see below (\ref{0.6}), and on how close $F^u_N$ is to a translate of $B_\nu$ (the closed Euclidean ball with volume $\nu$ centered at the origin). 

\begin{thm}\label{theo2.1} (upper bound)

There exists $\wt{L}_0(N)$ satisfying (\ref{0.3}) such that for any $0 < u < \ov{u}$ and $\nu > 0$ 
\begin{equation}\label{2.1}
\limsup\limits_N \; \mbox{\no $\frac{1}{N^{d-2}}$} \; \log \IP[\,|\wt{W}^u_N| \ge \nu N^d] \le - \mbox{\no $\frac{1}{d}$} \;(\sqrt{\ov{u}} - \sqrt{u})^2 \,{\rm cap}(B_\nu),
\end{equation}
and so that for any $\eta > 0$,
\begin{equation}\label{2.2}
\begin{array}{l}
\limsup\limits_N \; \mbox{\no $\frac{1}{N^{d-2}}$} \; \log \IP[\,|\wt{W}^u_N| \ge \nu N^d, \delta (F^u_N, B_\nu) \ge \mu] 
\\[1ex]
 \le - \mbox{\no $\frac{1}{d}$} \;(\sqrt{\ov{u}} - \sqrt{u})^2 \,\big({\rm cap}(B_\nu) + c_1 (\nu, \mu)\big),
\end{array}
\end{equation}
where $\delta(F^u_N,B_\nu)$ is defined in (\ref{0.10}), see also (\ref{1.11}).
\end{thm}

\begin{proof}
The proof heavily relies on a coarse graining procedure (see also Figure 1 below) to construct a porous interface similar to Section 4 of \cite{NitzSzni}, as well as on the capacity lower bound proved in Corollary 3.4 of the same reference that we recalled in (\ref{1.4}).

We will first specify $\wt{L}_0(N)$ satisfying (\ref{0.3}) and introduce some scales. For this purpose we consider rationals $\alpha > \beta > \gamma$ in $(0,\ov{u})$, $\wt{\ve}$ rational in $(0,1)$, $K$ integer such that, as below (4.11) of \cite{NitzSzni}, $K \ge c(\alpha, \beta, \gamma, \wt{\ve}) (\ge 100)$, where this constant corresponds to $c_4(\alpha, \beta, \gamma) \vee c_5(\wt{\ve}) \vee c_8 (\alpha, \beta, \gamma)$ in the notation of Theorem 2.3, Proposition 3.1, and Theorem 5.1 of  \cite{Szni17}. For any such rationals $\alpha, \beta, \gamma, \wt{\ve}$ and integer $K$, we select a sequence $\gamma_N$ in $(0,1]$ satisfying the conditions (4.18) of \cite{NitzSzni}, in particular $\gamma_N \underset{N}{\longrightarrow} 0$ (this choice can actually be performed independently of $\wt{\ve}$, but this feature will be irrelevant for us). We then set, see (4.19) of \cite{NitzSzni},
\begin{equation}\label{2.3}
L_0 = [(\gamma^{-1}_N \,N \log N)^{\frac{1}{d-1}}], \; \wh{L}_0 = 100 d\, [\sqrt{\gamma}_N\, N],
\end{equation}
so that, cf.~(4.24) of \cite{NitzSzni}
\begin{equation}\label{2.4}
\wh{L}_0/ L_0 \underset{N}{\longrightarrow} \infty.
\end{equation}
We then choose $\wt{L}_0(N)$ to be any non-negative integer valued sequence such that
\begin{equation}\label{2.5}
\left\{  \begin{array}{rl}
{\rm i)} & \wt{L}_0(N) = o(N),
\\[1ex]
{\rm ii)} &\wt{L}_0(N) \, / \, \wh{L}_0(N) \underset{N}{\longrightarrow} \infty \;\mbox{for any given choice of $\alpha, \beta, \gamma, \wt{\ve}$ }
\\
& \mbox{and $K$ (there are countably many such choices)}.
\end{array}\right.
\end{equation}
One can always construct such an $\wt{L}_0$ via a diagonal procedure (incidentally, a more explicit rate of decay of the function $\rho(L)$ depending on $\alpha, \beta, \gamma, K$ in (4.16) of \cite{NitzSzni}, see also (5.6) of \cite{Szni17}, would yield a more explicit $\wt{L}_0$). 

We then introduce the lattices (cf.~(4.20) of \cite{NitzSzni}):
\begin{equation}\label{2.6}
\IL_0 = L_0 \,\IZ^d, \; \wh{\IL}_0 = \mbox{\no $\frac{1}{100d}$} \; \wh{L}_0 \, \IZ^d = [\sqrt{\gamma}_N \,N] \, \IZ^d
\end{equation}
(with $[\cdot]$ the integer part), and for $z \in \IL_0$ we define, see (4.12) of \cite{NitzSzni},
\begin{equation}\label{2.7}
\begin{split}
B_z = z + [0,L_0)^d \cap \IZ^d &\subseteq D_z = z + [-3L_0, 4L_0)^d \cap \IZ^d 
\\
&\subseteq U_z =  z+ [-K L + 1, K L - 1)^d \cap \IZ^d.
\end{split}
\end{equation}
We now consider some $u$ in $(0,\ov{u})$, pick $\alpha > \beta > \gamma$ rationals in $(u,\ov{u})$, and $\wt{\ve}$ rational in $(0,1)$ so that $\wt{\ve} \,(\sqrt{\frac{\ov{u}}{u}} - 1) < 1$, as well as an integer $K \ge c(\alpha, \beta, \gamma, \wt{\ve}) ( \ge 100)$. We denote by $N_u(D_z)$ the number of excursions from $D_z$ to the exterior boundary $\partial U_z$ of $U_z$ that are contained in the interlacement trajectories up to level $u$, see (2.14) and (1.42) of \cite{Szni17}. We refer to (2.11)~-~(2.13) of \cite{Szni17} for the notion of a good$(\alpha, \beta, \gamma)$-box $B_z$ (which is otherwise bad$(\alpha, \beta, \gamma)$). The details of the definition will not be important here. Very roughly, one looks at the excursions of the interlacements between $D_z$ and the complement of $U_z$ (they are ordered in a natural fashion). For a good$(\alpha, \beta, \gamma)$-box $B_z$, the complement of the first $\alpha \,{\rm cap}_{\IZ^d}(D_z)$ excursions (with ${\rm cap}_{\IZ^d}(\cdot)$ the simple random walk capacity) contains in $B_z$ a connected set with $|\cdot |_\infty$-diameter at least $L_0/10$, which is connected to similar components in neighboring boxes $B_z$ via paths in $D_z$ avoiding the first $\beta \,{\rm cap}_{\IZ^d}(D_z)$ excursions in the interlacement. In addition, the first $\beta \,{\rm cap}_{\IZ^d}(D_z)$ excursions spend a substantial ``local time'' on the inner boundary of $D_z$, which is at least $\gamma \, {\rm cap}_{\IZ^d}(D_z)$.

Informally, the variables $N_u(D_z)$ aim at tracking an ``undertow'' in the medium, whereas the notion of good$(\alpha,\beta,\gamma)$-box relates to a spatially faster decorrelating information, see also Remark 2.2 and above Theorem 3.2 of \cite{Szni15}. In Section 4, this splitting between ``undertow'' and ``local'' behavior will correspond to the decomposition of the Gaussian free field $\varphi$ into an harmonic average $h_{B_z}$ inside $U_z$ and the local field $\psi_{B_z}$ (vanishing outside $U_z$) so that $\varphi  = h_{B_z} + \psi_{B_z}$, see below (\ref{4.3}).

We can then introduce as in (4.27) of \cite{NitzSzni} (with the choice $M=1$ in (4.27) of \cite{NitzSzni})
\begin{equation}\label{2.8}
\begin{split}
\cU^1 = &\; \mbox{the union of all $L_0$-boxes $B_z$ that are either contained in}
\\[-0.5ex]
&\;\mbox{$B(0,2N)^c$ or linked to an $L_0$-box contained in $B(0,2N)^c$}
\\[-0.5ex]
&\;\mbox{by a path of $L_0$-boxes $B_{z_i}$, $0 \le i \le n$, which are all, except}
\\
&\mbox{maybe for the last one, good$(\alpha,\beta,\gamma)$ and such that}
\\
&\mbox{$N_u(D_{z_i}) < \beta \, {\rm cap}_{\IZ^d} (D_{z_i})$}.
\end{split}
\end{equation}
We write
\begin{equation}\label{2.9}
\cH^u_N = \{\,|\wt{\cW}^u_N| \ge \nu N^d\}
\end{equation}
for the event that $\wt{\cC}^u_N$ (see (\ref{0.4})), leaves a ``hole'' $\wt{\cW}^u_N = B(0,N) \backslash \wt{\cC}^u_N$ in $B(0,N)$ containing at least $\nu N^d$ sites. In the present context, this ``hole event'' replaces the disconnection event $\cD^u_N$ from (4.9) of \cite{NitzSzni}. With (\ref{2.2}) in mind, we also introduce the event
\begin{equation}\label{2.10}
\cH^{u,\mu}_N = \{ \,|\wt{\cW}^u_N| \ge \nu N^d, \delta(F^u_N, B_\nu) \ge \mu\} \subseteq \cH^u_N,
\end{equation}
where in addition the $\IR^d$-filling $F^u_N$ of $\frac{1}{N} \,\wt{W}^u_N$, see (\ref{0.9}), has an indicator function at $L^1$-distance at least $\mu$ from all translates of the indicator function of $B_\nu$, see (\ref{1.11}). Then, as in (4.28) of \cite{NitzSzni}, we introduce the function
\begin{equation}\label{2.11}
\wh{\sigma}(x) = | \cU^1 \cap B(x,\wh{L}_0) \, | \, / \, |B(x,\wh{L}_0)|, \,x \in \IZ^d,
\end{equation}
and note that $\wh{\sigma}(\cdot)$ has slow variation, in the sense that
\begin{equation}\label{2.12}
|\wh{\sigma}(x + e) - \wh{\sigma}(x)| \le \mbox{\no $\frac{1}{\wh{L}_0}$}, \;\mbox{for all $x, e$ in $\IZ^d$ with $|e|_1 = 1$}
\end{equation}
(where $| \cdot |_1$ stands for the $\ell^1$-norm).

First note that when $B(x,\wh{L}_0 + L_0) \subseteq B(0,2N)^c$, any $L_0$-box $B_z$ intersecting $B(x,\wh{L}_0)$ is contained in $B(0,2N)^c$ and hence in $\cU^1$, so, as in (4.30) of \cite{NitzSzni},
\begin{equation}\label{2.13}
\wh{\sigma}(x) = 1, \;\mbox{when} \;B(x,\wh{L}_0 + L_0) \subseteq B(0,2N)^c .
\end{equation}
The next observation differs from \cite{NitzSzni} (compare with (4.31) in \cite{NitzSzni}), and reflects that we are interested in the ``hole event'' $\cH^u_N$ from (\ref{2.9}), or in the event $\cH^{u,\mu}_N \subseteq \cH^u_N$ from (\ref{2.10}), in place of the disconnection event $\cD^u_N$ from (4.9) of \cite{NitzSzni}. We will now show that
\begin{equation}\label{2.14}
\begin{array}{l}
\mbox{for large $N$, on $\cH^u_N$, $\wh{\sigma}(x) = 0$ when $x$ belongs to the }
\\
\mbox{$(\wt{L}_0 - \wh{L}_0 - L_0)$-neighborhood in $|\cdot |_\infty$-distance of $\wt{\cW}^u_N$}.
\end{array}
\end{equation}
For this purpose we will first establish that
\begin{equation}\label{2.15}
\mbox{for large $N$, the $(\wt{L}_0 - L_0)$-neighborhood of $\wt{\cW}^u_N$ does not meet $\cU^1$}.
\end{equation}
Indeed, observe first that $\wt{\cW}^u_N \subseteq B(0,N- \wt{L}_0 - 1)$ (because $\wt{\cC}^u_N$ contains the $\wt{L}_0$-neighborhood of $S_N \subseteq \cC^u_N$). If the $(\wt{L}_0 - L_0)$-neighborhood of $\wt{\cW}^u_N$ meets $\cU^1$, then there is an $L_0$-box $B_z$ in $\cU^1$, which meets the $(\wt{L}_0 - L_0)$-neighborhood of $\wt{\cW}^u_N$. Hence $B_z$ is contained in $B(0,N-\wt{L}_0 - 1 + \wt{L}_0 - L_0 + L_0) = B(0,N-1)$ and in $\cU^1$. By the connectivity statement in (4.13) of \cite{NitzSzni} (or Lemma 6.1 in \cite{Szni17}), $B_z$ contains a vertex $y$ that belongs to a connected component of $\cV^u$, which meets an $L_0$-box having a neighboring box contained in $B(0,2N)^c$. In particular, this connected component meets $S_N$, and the vertex $y$ belongs to $\cC^u_N$. In addition, since $B_z$ has $| \cdot |_\infty$-diameter $L_0 - 1$, and $B_z$ meets the $(\wt{L}_0 - L_0)$-neighborhood of $\wt{\cW}^u_N$, the vertex $y$ lies at distance $\wt{L}_0 - 1$ of $\wt{\cW}^u_N$, but this is a contradiction since $y \in \cC^u_N$. This proves (\ref{2.15}).

Let us now prove (\ref{2.14}). To this end we note that for large $N$, on $\cH^u_N$, when $x$ belongs to the $(\wt{L}_0 - \wh{L}_0 - L_0)$-neighborhood of $\wt{\cW}^u_N$, then $B(x,\wh{L}_0)$ is contained in the $(\wt{L}_0 - L_0)$-neighborhood of $\wt{\cW}^u_N$, and by (\ref{2.15}) it does not meet $\cU^1$, so that $\wh{\sigma}(x) = 0$. This proves (\ref{2.14}).

We now proceed in a similar fashion to (4.32) in \cite{NitzSzni}, and define
\begin{equation}\label{2.16}
\wh{\cS}_N = \big\{x \in \wh{\IL}_0; \; \mbox{\no $\frac{1}{4}$} \le \wh{\sigma}(x) \le \mbox{\no $\frac{3}{4}$}\big\},
\end{equation}
as well as the compact subset of $\IR^d$
\begin{equation}\label{2.17}
\Delta_N = \mbox{\no $\bigcup\limits_{x \in \wh{\cS}_N}$} B_{\IR^d} \big(\mbox{\no $\frac{x}{N}$}, \mbox{\no $\frac{1}{50d}$} \; \mbox{\no $\frac{\wh{L}_0}{N}$}\big),
\end{equation}
where for $z$ in $\IR^d$ and $r \ge 0$, $B_{\IR^d}(z,r)$ stands for the closed ball in $\IR^d$ with center $z$ and radius $r$ for the $| \cdot |_\infty$-distance.

In the present context, the following lemma replaces Lemma 4.3 of \cite{NitzSzni}.

\begin{lem}\label{lem2.2} (insulation property of $\Delta_N$)

For large $N$,
\begin{equation}\label{2.18}
\wh{\cS}_N \subseteq B(0,3N) \cap \wh{\IL}_0,
\end{equation}
and on $\cH^u_N$,
\begin{align}
&\mbox{the closed $\frac{\wt{L}_0}{2 N}$-neighborhood in $\IR^d$ of $\frac{1}{N} \,\wt{\cW}^u_N$ for the}\label{2.19}
\\
&\mbox{$| \cdot |_\infty$-distance is contained in the union of the bounded} \nonumber
\\
&\mbox{components of the open set $\IR^d \backslash \Delta_N$}. \nonumber
\end{align}
\end{lem}

\begin{proof}
The proof of (\ref{2.18}) is the same as the proof of (4.34) of \cite{NitzSzni}: by (\ref{2.13}) we know that $\wh{\sigma}(x) = 1$, when $B(x,\wh{L}_0 + L_0) \subseteq B(0,2 N)^c$. Hence, for large $N$, when $|x|_\infty \ge 3N$, one has $B(x,\wh{L}_0 + L_0) \subseteq B(0,2N)^c$, so that $\wh{\sigma}(x) = 1$. The claim (\ref{2.18}) follows.

We now turn to the proof of (\ref{2.19}). We will use the fact that
\begin{align}
&\mbox{for large $N$, on $\cH^u_N$, any continuous $\Psi$: $[0,1] \r \IR^d$ such that $\Psi(0)$}\label{2.20}
\\[-1ex]
&\mbox{is within $| \cdot |_\infty$-distance $1$ of $\{x \in \IZ^d$; $d_\infty(x,\wt{\cW}^u_N) \le \frac{\wt{L}_0}{2}\}$ and such} \nonumber
\\[-0.6ex]
&\mbox{that $|\Psi(1)|_\infty \ge 3N$, comes within $|\cdot |_\infty$-distance $\frac{1}{50d} \;\wh{L}_0$ from $\wh{\cS}_N$}. \nonumber
\end{align}
The proof is very similar to that of (4.37) of \cite{NitzSzni}. Indeed, given $\Psi$ as above, one constructs a $\IZ^d$-valued $*$-path $y_i$, $0 \le i \le \ell$ (i.e. $|y_{i+1} - y_i|_\infty = 1$, for $0 \le i < \ell)$, such that $d_\infty(y_0, \wt{\cW}^u_N) \le \frac{\wt{L}_0}{2}$ and $|y_\ell |_\infty > 2N + \wh{L}_0 + L_0$, and $y_0,y_1,\dots,y_\ell$ are contained in the closed $1$-neighborhood for the $| \cdot |_\infty$-distance of $\Psi ([0,1])$. By (\ref{2.14}) and (\ref{2.13}), we see that $\wh{\sigma}(y_0) = 0$ and $\wh{\sigma}(y_\ell) = 1$. Note that $|y_{i+1} - y_i|_1 \le d$, for $0 \le i < \ell$. So, by the Lipschitz property (\ref{2.12}) of $\wh{\sigma}$, we can find some $0 \le j \le \ell$ with $|\wh{\sigma}(y_j) - \frac{1}{2}| \le \frac{d}{\wh{L}_0}$. Then, for large $N$, if we choose $\wh{y} \in \wh{\IL}_0$ with $| \wh{y} -y_j|_\infty \le \frac{1}{100d} \;\wh{L}_0$, we find that
\begin{equation}\label{2.21}
\big|\wh{\sigma}(\wh{y}) - \mbox{\no $\frac{1}{2}$} \big| \le  \mbox{\no $\frac{1}{\wh{L}_0}$} \;|\wh{y} - y_j|_1 +   \mbox{\no $\frac{d}{\wh{L}_0}$} \le   \mbox{\no $\frac{1}{100}$} +  \mbox{\no $\frac{d}{\wh{L}_0}$} <  \mbox{\no $\frac{1}{4}$} .
\end{equation}
This proves that $\wh{y} \in \wh{\cS}_N$, cf.~(\ref{2.16}). Observe that $\Psi(\cdot)$ comes within $| \cdot |_\infty$-distance $\frac{1}{100d} \;\wh{L}_0 + 1 \le \frac{1}{50d} \;\wh{L}_0$ from $\wh{y}$. This proves (\ref{2.20}).

Let us conclude the proof of (\ref{2.19}). To this end, note that any point in $\IR^d$ in the $\frac{\wh{L}_0}{2N}$-closed neighborhood of $\frac{1}{N} \,\wt{\cW}^u_N$ for the $| \cdot |_\infty$-distance, is within $| \cdot |_\infty$-distance $\frac{1}{N}$ of some point in $\frac{1}{N} \,\IZ^d$ in the $\frac{\wt{L}_0}{2N}$-closed neighborhood of $\frac{1}{N} \, \wt{\cW}^u_N$ for the $| \cdot |_\infty$-distance. So, by (\ref{2.20}), any continuous path in $\IR^d$ starting in the $\frac{\wt{L}_0}{2N}$-closed neighborhood of $\frac{1}{N} \, \wt{\cW}^u_N$ for the $| \cdot |_\infty$-distance and tending to infinity necessarily encounters $\Delta_N$. The claim (\ref{2.19}) follows. This concludes the proof of Lemma \ref{lem2.2}.
\end{proof}

We then proceed as in (4.39) of \cite{NitzSzni}. We go over the main steps. We extract a random subset $\wt{\cS}_N$ of $\wh{\cS}_N$ such that
\begin{equation}\label{2.22}
\begin{array}{l}
\mbox{$\wt{\cS}_N$ is a maximal subset of $\wh{\cS}_N$ such that the $B(x,2 \wh{L}_0), \,x \in \wt{\cS}_N$,}
\\
\mbox{are pairwise disjoint.}
\end{array}
\end{equation}
We introduce the notation $N_{L_0} = L_0^{d-1} / \log L_0$ and the bad event, see (4.22) of \cite{NitzSzni} (with $(e_i)_{1 \le i \le d}$ the canonical basis of $\IR^d$)
\begin{equation}\label{2.23}
\begin{split}
\cB_N = \mbox{\normalsize $\bigcup\limits_{e\in \{e_1,\dots, e_d\}}$} \;\; \big\{ &\mbox{there are at least $\rho(L_0) ( \mbox{\no $\frac{N_{L_0}}{L_0}$})^{d-1}$ columns of}
\\[-2ex]
&  \mbox{$L_0$-boxes in the direction $e$ in $B(0,20 N)$ that}
\\
& \mbox{contain a bad$(\alpha, \beta, \gamma)$ $L_0$-box$\big\}$}
\end{split}
\end{equation}
(with $\rho(L)$ a suitable function depending on $\alpha, \beta, \gamma$, and $K$, which tends to $0$ as $L \r \infty$).

By Lemma 4.2 of \cite{NitzSzni}, we know that
\begin{equation}\label{2.24}
\rho(L_0) \;\Big( \mbox{\no $\frac{N_{L_0}}{L_0}$}\Big)^{d-2} / \;(\wh{L}_0 /  L_0)^{d-1} \underset{N}{\longrightarrow} 0,
\end{equation}
and that the following super-exponential bound holds
\begin{equation}\label{2.25}
\lim\limits_N \;  \mbox{\no $\frac{1}{N^{d-2}}$} \;\log \IP[\cB_N] = - \infty.
\end{equation}
So the bad even $\cB_N$ in (\ref{2.23}) is ``negligible'' for our purpose. We thus introduce the effective events
\begin{equation}\label{2.26}
\wt{\cH}^u_N = \cH^u \backslash \cB_N \supseteq \wt{\cH}^{u,\mu}_N = \cH^{u,\mu}_N \backslash  \cB_N.
\end{equation}
As in (4.41) of \cite{NitzSzni}, setting $\ov{K} = 2 K + 3$, we find that
\begin{equation}\label{2.27}
\begin{array}{l}
\mbox{for large $N$, on $\wt{\cH}_N^u$, for each $x \in \wt{\cS}_N$ there is a collection $\wt{\cC}_x$}
\\
\mbox{of $L_0$-boxes intersecting $B(x,\wh{L}_0)$ with $\wt{\pi}_x$-projection at mutual}
\\
\mbox{distance at least $\ov{K} L_0$ and cardinality $[(\frac{c'}{K} \; \frac{\wh{L}_0}{L_0})^{d-1}]$ such that for}
\\
\mbox{each $z \in \wt{\cC}_x$, $B_z$ is good$(\alpha, \beta, \gamma)$ and $N_u(D_z) \ge \beta \, {\rm cap}_{\IZ^d}(D_z)$}
\end{array}
\end{equation}
(for each $x \in \wt{\cS}_N$, one has $\wt{i}_x$ in $\{1,\dots,d\}$ and  $\wt{\pi}_x$ denotes the projection on the set of points in $\IZ^d$ with vanishing $\wt{i}_x$-coordinate).

As below (4.41) of \cite{NitzSzni}, for large $N$, we can define a random variable on $\wt{\cH}^u_N$
\begin{equation}\label{2.28}
\kappa_N = (\wh{\cS}_N,\wt{\cS}_N, (\wt{\pi}_x, \wt{\cC}_x)_{x \in \wt{\cS}_N}),
\end{equation}
with range $\cK_N$, which is a set such that, cf.~(4.43) of \cite{NitzSzni},
\begin{equation}\label{2.29}
|\cK_N| = \exp\{ o(N^{d-2})\}.
\end{equation}
This provides for large $N$ a coarse graining of the event $\wt{\cH}^u_N$ in (\ref{2.26}) (see Figure 1). Namely, for large $N$ one has the partition
\begin{equation}\label{2.30}
\wt{\cH}^u_N =  \mbox{\no $\bigcup\limits_{\kappa \in \cK_N}$} \;\cH_{N,\kappa} \;\;\mbox{where} \;\; \cH_{N,\kappa} = \wt{\cH}^u_N \cap \{\kappa_N = \kappa\}.
\end{equation}
As in (4.47)--(4.49) of \cite{NitzSzni}, for each $\kappa \in \cK_N$, we associate a ``segmentation'' corresponding to $U_0$ or $S$ in (\ref{2.31}) below, and a porous interface corresponding to $\Sigma$ in (\ref{2.31}) below, via:
\begin{equation}\label{2.31}
\left\{ \begin{array}{rl}
C = &\!\!\! \bigcup\limits_{x \in \wt{\cS}} \; \bigcup\limits_{z \in \wt{\cC}_x} \,B_z (\subseteq \IZ^d),
\\
\\[-2ex]
\Sigma = &\!\!\! \frac{1}{N} \;\big(\bigcup\limits_{x \in \wt{\cS}} \; \bigcup\limits_{z \in \wt{\cC}_x} z + [0,L_0]^d\big) (\subseteq \IR^d),
\\
\\[-2ex]
U_1 = &\!\!\!\mbox{the unbounded component of}
\\
&\!\!\!\mbox{$\IR^d  \backslash  \big\{ \frac{1}{N} \big(\bigcup\limits_{x \in \wh{\cS}} B_{\IR^d}\big(x, \frac{1}{50d} \;\wh{L}_0\big)\big)\big\}$},
\\
\\[-2ex]
U_0 = &\!\!\! \IR^d \backslash U_1, \,S = \partial U_0 = \partial U_1.
\end{array}\right.
\end{equation}

\psfrag{AN}{$A_N$}
\psfrag{SN}{$S_N$}
\psfrag{WN}{$\wt{\cW}^u_N$}
\psfrag{dist}{in $\cU^1$}
\psfrag{L}{$| \cdot |_\infty$-distance at least $\wt{L}_0 / 2 \gg \wh{L}_0  \gg L_0$}
\psfrag{L0}{\small{$2\widehat{L}_0$}}
\psfrag{Lsk}{\small{$L_0 \gg (N \log N)^{\frac{1}{d-1}}$ ``slightly''}}
\includegraphics[width=8cm]{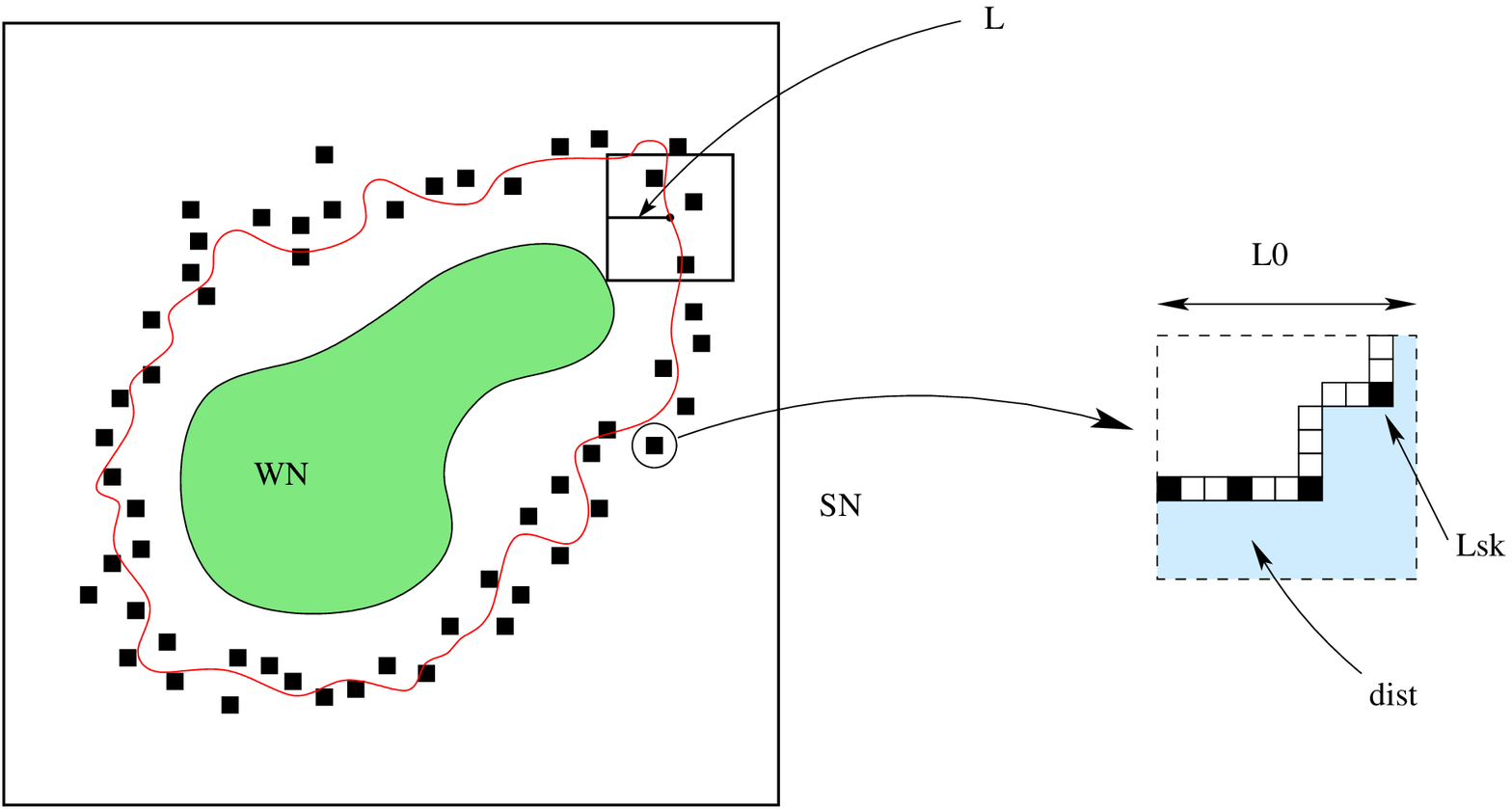}

\medskip
\begin{center}
\begin{tabular}{lp{10cm}}
Fig.~1: & An informal illustration of some features entering the definition of the event $\cH_{N,\kappa}$ corresponding to the coarse graining of the event $\wt{\cH}^u_N$ in (\ref{2.26}), with the selected boxes of side-length $2\wh{L}_0$ on the left-hand side, and the blow-up of one such box with the selected boxes of side-length $L_0$ (in black), on the right-hand side.
 \end{tabular}
\end{center}
We note that on $\cH_{N,\kappa}$ the above defined $U_1$ coincides with the unbounded component of $\IR^d \backslash \Delta_N$ in Lemma \ref{lem2.2}. Thus by (\ref{2.19}) 
\begin{equation}\label{2.32}
\left\{ \begin{array}{l}
\mbox{for large $N$ and any $\kappa \in \cK_N$, on $\cH_{N, \kappa}$, the $\frac{\wt{L}_0}{2N}$-closed}
\\
\mbox{neighborhood of $\frac{1}{N} \; \wt{\cW}^u_N$ for the $| \cdot |_\infty$-distance is contained}
\\
\mbox{in the union of the bounded components of the open set}
\\
\mbox{$\IR^d \backslash \big\{\frac{1}{N} \,\big(\bigcup\limits_{x \in \wh{\cS}} B_{\IR^d}\big(x , \frac{1}{50d}\;\wh{L}_0\big)\big)\big\}$}.
\end{array}\right.
\end{equation}
The proof now differs from \cite{NitzSzni} (see below (4.49) of that reference). We can now introduce for large $N$ the (deterministic) compact subsets of $\IR^d$
\begin{equation}\label{2.33}
A_\kappa = \big\{z \in U_0; \,d_\infty(z,U_1) \ge  \mbox{\no $\frac{\wt{L}_0}{4N}$}\big\}, \; \mbox{for $\kappa \in \cK_N$}.
\end{equation}
Recall the notation $F^u_N$ from (\ref{0.9}), for the $\IR^d$-filling of $\frac{1}{N} \;\wt{\cW}^u_N$. By (\ref{2.32}) and the definition of the hole event $\cH^u_N$ in (\ref{2.9}), we see that
\begin{equation}\label{2.34}
\begin{array}{l}
\mbox{for large $N$ and any $\kappa \in \cK_N$, on $\cH_{N, \kappa}$, one has $F^u_N \subseteq \mathring{A}_\kappa$ and}
\\
\nu \le |F^u_N| \le |\mathring{A}_\kappa | \;\mbox{(with $\mathring{A}_\kappa$ the interior of $A_\kappa$ and $|E|$ the volume}
\\
\mbox{of $E$ Borel subset of $\IR^d$}).
\end{array}
\end{equation}
In view of the definition of the Fraenkel asymmetry, see (\ref{1.9}), we then find by the triangle inequality that
\begin{equation}\label{2.35}
\begin{array}{l}
\mbox{for large $N$ and any $\kappa \in \cK_N$, on $\cH_{N, \kappa}$, $|\mathring{A}_\kappa \backslash F^u_N| = | \mathring{A}_\kappa | - |F^u_N|$,}
\\
\mbox{and} \; |\mathring{A}_\kappa | - |F^u_N| + |\mathring{A}_\kappa | - \nu + |\mathring{A}_\kappa | \,\lambda_{\mathring{A}_\kappa} \ge \delta (F^u_N,B_\nu).
\end{array}
\end{equation}
If we now set in the notation of (\ref{2.26}), for large $N$,
\begin{equation}\label{2.36}
\cK^\mu_N = \{\kappa\in \cK_N; \,\cH_{N,\kappa} \cap \cH^{u,\mu}_N \not= \phi\}.
\end{equation}
We now find by (\ref{2.35}) and (\ref{2.34}) that
\begin{equation}\label{2.37}
\begin{array}{l}
\mbox{for large $N$ and any $\kappa \in \cK^\mu_N$,}
\\
\mbox{$|\mathring{A}_\kappa | \ge \nu$ and $2 (|\mathring{A}_\kappa | - \nu) + |\mathring{A}_\kappa | \,\lambda_{\mathring{A}_\kappa} \ge \mu$}.
\end{array}
\end{equation}
The following lemma will be helpful in the proof of (\ref{2.2}) (and similarly in the proof of (\ref{4.2}) in Theorem \ref{theo4.1} of Section 4).

\begin{lem}\label{lem2.3} 
(coercivity of the capacity, see (\ref{1.9}) and below (\ref{0.8}) for notation)

Consider $\nu > 0$, $\mu > 0$ and $U$ a bounded open set in $\IR^d$ with
\begin{equation}\label{2.38}
|U| \ge \nu \;\;\mbox{and} \;\; 2(|U| - \nu) + |U| \,\lambda_U \ge \mu.
\end{equation}
Then, one has
\begin{equation}\label{2.39}
{\rm cap}(U) \ge {\rm cap}(B_\nu) + c_1(\nu,\mu).
\end{equation}
\end{lem}

\begin{proof}
If (\ref{2.38}) holds, then either
\begin{align}
&|U| - \nu \ge  \mbox{\no $\frac{\mu}{4}$} , \;\; \mbox{or} \label{2.40}
\\[1ex]
& |U| - \nu <  \mbox{\no $\frac{\mu}{4}$} \;\; \mbox{and}  \;\; |U|\, \lambda_U \ge  \mbox{\no $\frac{\mu}{2}$}. \label{2.41}
\end{align}

If (\ref{2.40}) holds, then $\eta_U \ge 0$, see (\ref{1.8}) for notation, and one has
\begin{equation}\label{2.42}
{\rm cap}(U) \ge {\rm cap}(B_{\nu  + \frac{\mu}{4}}) \ge {\rm cap}(B_\nu) + c(\nu,\mu).
\end{equation}
On the other hand, if (\ref{2.41}) holds, then 
\begin{equation}\label{2.43}
 \mbox{\no $\frac{\mu}{2}$} \stackrel{\mbox{\f (\ref{2.41})}}{\le} |U| \,\lambda_U \stackrel{\mbox{\f (\ref{2.41})}}{\le}  \big( \nu +  \mbox{\no $\frac{\mu}{4}$}\big)\,\lambda_U \stackrel{\mbox{\f (\ref{1.10})}}{\le} c_0\,\big(\nu +  \mbox{\no $\frac{\mu}{4}$}\big) \; \eta^{1/4}_U,
\end{equation}
so that
\begin{equation}\label{2.44}
\eta_U \ge \wt{c}(\nu,\mu).
\end{equation}
We then find that by definition of $\eta_U$ in (\ref{1.8}), that with $B$ a ball of volume $|U|$
\begin{equation}\label{2.45}
\begin{array}{l}
{\rm cap}(U) \ge \big(1 + \wt{c}(\nu,\mu)\big) \,{\rm cap}(B) \stackrel{\mbox{\f (\ref{2.38})}}{\ge} 
\\
\big(1 + \wt{c}(\nu,\mu)\big) \,{\rm cap}(B_\nu) \ge {\rm cap}(B_\nu) + c'(\nu,\mu).
\end{array}
\end{equation}
Collecting (\ref{2.42}) and (\ref{2.45}), we find (\ref{2.39}). This proves Lemma \ref{lem2.3}.
\end{proof}

We now resume the proof of (\ref{2.1}) and (\ref{2.2}). We begin with (\ref{2.1}). We proceed in a similar fashion to (4.52) of \cite{NitzSzni} and find with the exponential bound stated in (4.14) of \cite{NitzSzni} that in the notation of (\ref{2.31})
\begin{equation}\label{2.46}
\begin{array}{l}
\limsup\limits_N \; \frac{1}{N^{d-2}} \;\log \IP[\cH^u_N] \le - \Gamma \; \liminf_N \; \inf\limits_{\kappa \in \cK_N} \; \frac{1}{N^{d-2}} \;{\rm cap}_{\IZ^d}(C), 
\\[1ex]
\mbox{where}\; \Gamma = \big(\sqrt{\gamma} - \frac{\sqrt{u}}{1 - \wt{\ve}\,(\sqrt{\frac{\ov{u}}{u}} - 1)}\big) (\sqrt{\gamma} - \sqrt{u}).
\end{array}
\end{equation}
Then, taking a $\liminf$ over $K$ tending to infinity, we find with help of Proposition A.1 of the Appendix of \cite{NitzSzni} that (see (\ref{2.31}) for notation)
\begin{equation}\label{2.47}
\limsup\limits_N \; \mbox{\no $\frac{1}{N^{d-2}}$} \; \log \IP[\cH^u_N] \le - \Gamma \;\underset{K}{\underline{\lim}} \;\; \underset{N}{\underline{\lim}} \; \inf\limits_{\kappa \in \cK_N} \, \mbox{\no $\frac{1}{d}$} \;{\rm cap}(\Sigma).
\end{equation}
Now, by the capacity lower bound of Corollary 3.4 of \cite{NitzSzni} recalled in (\ref{1.4}) (where we choose $A = A_\kappa$, $\ve = 10 \,\frac{\wh{L}_0}{N}$, $\ell_*$ the smallest non-negative integer so that $2^{-\ell_*} \le \frac{\wt{L}_0}{10 N}$), we find
\begin{equation}\label{2.48}
\begin{split}
\limsup\limits_N \; \mbox{\no $\frac{1}{N^{d-2}}$} \; \log \IP[\cH^u_N]  &\le - \Gamma \;  \underset{K}{\underline{\lim}} \;\; \underset{N}{\underline{\lim}} \; \inf\limits_{\kappa \in \cK_N} \; \mbox{\no $\frac{1}{d}$} \;{\rm cap}(A_\kappa)
\\
& \le - \Gamma \; \mbox{\no $\frac{1}{d}$}  \;{\rm cap}(B_\nu),
\end{split}
\end{equation}
where we used (\ref{2.34}) and Polya-Szeg\"o's Inequality (i.e.~$\eta_U \ge 0)$ to find that, for large $N$ and any $\kappa \in \cK_N$, ${\rm cap}(A_\kappa) \ge {\rm cap}(\mathring{A}_\kappa) \ge {\rm cap}(B_\nu)$. We can then let $\wt{\ve}$ go to $0$ and $\alpha, \beta, \gamma$ go to $\ov{u}$ along rationals (these parameters enter $\Gamma$) to find our claim (\ref{2.1}).

We now turn to the proof of (\ref{2.2}). The same argument now yields 
\begin{equation}\label{2.49}
\limsup\limits_N \; \mbox{\no $\frac{1}{N^{d-2}}$} \; \log \IP[\cH^{u,\mu}_N]  \le - \Gamma \;  \underset{K}{\underline{\lim}} \;\; \underset{N}{\underline{\lim}} \; \inf\limits_{\kappa \in \cK^\mu_N} \; \mbox{\no $\frac{1}{d}$} \;{\rm cap}(A_\kappa),
\end{equation}
where $\cK^\mu_N$ is defined in (\ref{2.36}).

We can then apply (\ref{2.37}) and Lemma \ref{2.3} to find that for large $N$ and any $\kappa \in \cK^\mu_N$, one has ${\rm cap}(A_\kappa) \ge {\rm cap}(\mathring{A}_\kappa) \ge {\rm cap}(B_\nu) + c_1 (\nu, \mu)$. Inserting this information in (\ref{2.49}) now yields
\begin{equation}\label{2.50}
\limsup\limits_N \; \mbox{\no $\frac{1}{N^{d-2}}$} \; \log \IP[\cH^{u,\mu}_N]  \le - \Gamma  \; \big(\mbox{\no $\frac{1}{d}$} \;{\rm cap}(B_\nu) + c_1(\nu,\mu)\big).
\end{equation}
Letting $\wt{\ve}$ tend to $0$ and $\alpha,\beta,\gamma$ tend to $\ov{u}$ along rationals concludes the proof of (\ref{2.2}), and hence of Theorem \ref{theo2.1}.
\end{proof}

We now turn to the asymptotic lower bound. The restriction $\nu < \omega_d$ permits us to avoid boundary effects in the minimization problem, where we look for a set of minimal capacity among all compact subsets $A$ of $\mathring{B}(0,1)$ with volume $\nu$: we have a minimum for the choice $A = B_\nu \subset \mathring{B}(0,1)$. Our main result is

\begin{thm}\label{theo2.4} (lower bound)

For any $0 < u < u_{**}$, $0 < \nu < \omega_d$, and $\wt{L}_0(N)$ satisfying (\ref{0.3}), we have
\begin{equation}\label{2.51}
\liminf\limits_N \;\;\mbox{\no $\frac{1}{N^{d-2}}$} \;\log \IP[\,|\wt{\cW}^u_N | \ge \nu N^d] \ge - \mbox{\no $\frac{1}{d}$} \;(\sqrt{u}_{**} - \sqrt{u})^2 \;{\rm cap}(B_\nu)
\end{equation}
(see (\ref{0.6}) and below (\ref{0.8}) for notation).
\end{thm}

\begin{proof}
We will in essence apply the construction and the results of \cite{LiSzni14}, and introduce a probability $\wt{\IP}_N$ corresponding to so-called {\it tilted interlacements}, under which, with probability tending to $1$ as $N$ goes to infinity, $\wt{\cC}^u_N$ in (\ref{0.4}) does not meet the Euclidean ball centered at the origin and volume $\nu' \,N^d$ (with $\nu'$ slightly larger than $\nu$). Then Theorem \ref{theo2.4} will follow with the change of probability method (i.e.~the inequality (\ref{1.6})) and the relative entropy bounds from Proposition 2.3 and 2.4 of \cite{LiSzni14}.

We thus consider $0 < u < u_{**}$ and $0 < \nu < \omega_d$ and define $R_\nu \in (0,1)$ via 
\begin{equation}\label{2.52}
\nu = \omega_d \;R^d_\nu
\end{equation}
(i.e.~the Euclidean ball $B_\nu$ from (\ref{0.8}) has radius $R_\nu$). We choose $\delta \in (0,1)$ such that
\begin{equation}\label{2.53}
R_\nu + 2 \delta < 1,
\end{equation}
as well as $\ve \in (0,1)$ and $r > 1$. We consider the function $h(z)$ on $\IR^d$ solution of the equilibrium problem
\begin{equation}\label{2.54}
\left\{ \begin{split}
\Delta h = 0 &  \;\;\mbox{in $U \backslash B_2(0,R_\nu + 2 \delta), \;\mbox{with} \; U = \mathring{B}_2(0,r)$}
\\
h = 1 &\;  \;\mbox{on $B_2(0,R_\nu + 2 \delta)$, and $h = 0$ in $\IR^d \backslash U$},
\end{split}\right.
\end{equation}
where we recall that $B_2(z,a)$ stands for the closed Euclidean ball in $\IR^d$ with center $z$ and radius $a$, and $\mathring{B}_2(z,a)$ for the corresponding open ball. As an aside, note that $h(z)$ can be represented as the probability that Brownian motion starting at $z$ enters $B_2(0,R_\nu + 2 \delta)$ before exiting $U$.

We then choose $0 < \eta < \delta$ and a non-negative smooth function $\phi^\eta$ with support in $B_2(0,\eta)$ such that $\int_{\IR^d} \,\phi^\eta(z) \, dz = 1$. We regularize $h$ by convolution with $\phi^\eta$ and set
\begin{equation}\label{2.55}
h^\eta = h * \phi^\eta .
\end{equation}
We then define for $N \ge 1$,
\begin{equation}\label{2.56}
f_N(x) = \Big(\sqrt{\mbox{\no $\frac{u_{**} + \ve}{u}$}} - 1 \Big) \;h^\eta \big(\mbox{\no $\frac{x}{N}$}\big) + 1 \;\; \mbox{for $x \in \IZ^d$},
\end{equation}
and introduce as in (2.7) of \cite{LiSzni14}
\begin{equation}\label{2.57}
\begin{array}{l}
\wt{\IP}_N = e^{F_{u,N}} \IP, \;\mbox{with}
\\
F_{u,N} (\omega) = \sum\limits_{i \ge 0} \;\int_\IR - \mbox{\no $\frac{\Delta_{\rm disc} f_N}{f_N}$} \;\big(X_s(w_{i,u})\big) \,ds,
\end{array}
\end{equation}
where $w_{i,u}$, $i \ge 0$, keeps track of all doubly infinite continuous-time $\IZ^d$-valued trajectories in the interlacement up to level $u$, $X_s(w_{i,u})$ stands for the position of $w_{i,u}$ at time $s$, and
\begin{equation}\label{2.58}
(\Delta_{\rm disc} g)(x) = \mbox{\no $\frac{1}{2d}$} \;\mbox{\normalsize $\sum\limits_{|e|_1 = 1}$} \;\big(g(x+ e) - g(x)\big), \; x \in \IZ^d,
\end{equation}
is the discrete Laplacian of $g$, for $g$: $\IZ^d \r \IR$.

As shown in Proposition 2.1 of \cite{LiSzni14}, $\wt{\IP}_N$ is a probability measure. We then consider $\Gamma_N$ as above (3.1) of \cite{LiSzni14}, the exterior boundary of the discrete blow-up of $B_2(0,R_\nu + \frac{\delta}{2})$, namely
\begin{equation}\label{2.59}
\Gamma_N = \partial \big\{x \in \IZ^d; \mbox{$x$ is within $|\cdot |_\infty$-distance $1$ \!of $B_2\big(0,\big(R_\nu +  \mbox{\normalsize $\frac{\delta}{2}$}\big)\,N\big)\!\big\}$}.
\end{equation}
By (4.10) of \cite{LiSzni14}, we know that
\begin{equation}\label{2.60}
\wt{\IP}_N [ \cA^u_N] \underset{N}{\longrightarrow} 1, \;\; \mbox{if} \;\; \cA^u_N = \mbox{\normalsize $\bigcap\limits_{x \in \Gamma_N}$} \big\{ x \stackrel{\cV^u}{\mbox{\Large $\longleftrightarrow$}} \hspace{-3.5ex} /\quad S_{N-1}\big\}
\end{equation}
(i.e.~$\cA^u_N$ is the event stating that no site of $\Gamma_N$ is linked by a path in $\cV^u$ to a site with $| \cdot |_\infty$-norm $N-1$).

We now choose an arbitrary sequence of non-negative integers $\wt{L}_0 (N)$ satisfying (\ref{0.3}). Then for large $N$, on $\cA^u_N$ one has
\begin{equation*}
\cC^u_N \subseteq \IZ^d \backslash B_2\big(0,\big(R_\nu + \mbox{\no $\frac{\delta}{2}$}\big) \;N\big),  \;\;\mbox{so that} \;\;\wt{\cW}^u_N \supseteq B_2 \big(0,\big(R_\nu + \mbox{\no $\frac{\delta}{4}$}\big) \;N\big) \cap \IZ^d.
\end{equation*}
We thus find that
\begin{equation}\label{2.61}
\mbox{for large $N$, on $\cA^u_N$, one has $| \,\wt{\cW}^u_N| \ge \nu N^d$}.
\end{equation}
By the change of probability method, cf.~(\ref{1.6}), we have
\begin{equation}\label{2.62}
\IP [\cA^u_N] \ge \wt{\IP}_N [\cA^u_N] \;\exp\big\{- \mbox{\no $\frac{1}{\wt{\IP}_N[\cA^u_N]}$} \big(H(\wt{\IP}_N \, | \, \IP) + \mbox{\no $\frac{1}{e}$}\big)\big\},
\end{equation}
and by (\ref{2.60}), (\ref{2.61}), we find that
\begin{equation}\label{2.63}
\liminf\limits_N \;  \mbox{\no $\frac{1}{N^{d-2}}$} \; \log \IP [ \,|\wt{\cW}^u_N| \ge \nu N^d] \ge - \limsup_N \;  \mbox{\no $\frac{1}{N^{d-2}}$}  \; H(\wt{\IP}_N \,| \, \IP).
\end{equation}
Then, by Propositions 2.3 and 2.4 of \cite{LiSzni14}, letting successively $\eta$ tend to $0$, $r$ to $\infty$, and $\delta$ tend to zero, we obtain
\begin{equation}\label{2.64}
\liminf\limits_N \;  \mbox{\no $\frac{1}{N^{d-2}}$} \; \log \IP [ \,|\wt{\cW}^u_N| \ge \nu N^d] \ge - \mbox{\no $\frac{1}{d}$}  \;\big(\sqrt{u_{**} + \ve} - \sqrt{u}\big)^2 \,{\rm cap}(B_\nu).
\end{equation}
Letting now $\ve$ tend to $0$ yields our claim (\ref{2.51}).
\end{proof}

\begin{rmk}\label{rem2.5} \rm  ~

1) Let us mention that the tilted interlacements provide a kind of slowly space-modulated random interlacements with space dependent parameter $u \,f_N(\frac{x}{N})^2$, see the Introduction of \cite{LiSzni14}. Actually, see below (1.49) of \cite{LiSzni14}, the law of $\cI^u$ under $\wt{\IP}_N$ coincides with the law of the interlacement at level $1$ if one endows each nearest neighbor edge $\{x,y\}$ of $\IZ^d$ with the conductance $\frac{u}{2d}\, f_N(\frac{x}{N}) \, f_N (\frac{y}{N}$). We will however not need this fact here.

2) Let us also point out that for large $N$, under the tilted interlacement measure $\wt{\IP}_N$, the vacant set $\cV^u$ inside $B_2(0,R_\nu N) \cap \IZ^d$ is in a strongly non-percolative regime (see Proposition 4.1 of \cite{LiSzni14}). As a result with overwhelming probability under $\wt{\IP}_N$ as $N$ goes to infinity, there are no macroscopic components in $\cV^u \cap B_2(0,R_\nu N)$. This feature is very different from what would happen in the case of Bernoulli percolation in the strongly percolative regime. A similar constraint like $|\wt{\cW}^u_N| \ge \nu N^d$ would qualitatively be produced by the creation of a blocking interface with a Wulff-shape like aspect (see for instance Theorem 2.12 of \cite{Cerf00}), but the percolative regime would be preserved inside the blocking interface.

3) If the equalities $\ov{u} = u_* = u_{**}$ hold, then choosing $\wt{L}_0(N)$ as in the Theorem \ref{theo2.1}, $0 < u < u_*$ and $\nu < \omega_d$, the upper and lower bounds (\ref{2.1}) and (\ref{2.51}) are matching so that (with $B_\nu$ as below (\ref{0.8}))
\begin{equation}\label{2.65}
\lim\limits_N \;  \mbox{\no $\frac{1}{N^{d-2}}$} \; \log \IP [ \,|\wt{\cW}^u_N| \ge \nu N^d] = - \mbox{\no $\frac{1}{d}$}  \;\big(\sqrt{u}_* - \sqrt{u})^2 \,{\rm cap}(B_\nu).
\end{equation}
In addition, by (\ref{2.2}) and (\ref{2.51}), conditionally on $|\wt{\cW}^u_N| \ge \nu N^d$, the $\IR^d$-filling $F^u_N$ of $\frac{1}{N} \; \wt{\cW}^u_N$, is close to a translate of $B_\nu$ in the sense that
\begin{equation}\label{2.66}
\lim\limits_N \;\IE[\delta(F^u_N, B_\nu) \,\big| \, |\wt{\cW}^u_N| \ge \nu N^d] = 0
\end{equation}
(with $\delta(\cdot,\cdot)$ as in (\ref{1.11})).

However, coming back to 2) above, it is unclear what happens for large $N$  under the conditional measure $\IP[\,\, \cdot \, \, \big| \, |\wt{\cW}^u_N| \ge \nu N^d]$ inside the ``nearly spherical hole'' $\wt{\cW}^u_N$ (left in $B(0,N)$ by $\wt{\cC}^u_N$). Are large connected components of $\cV^u$ present or not? The picture ought to be nearly critical, and possibly quite different from what takes place under the measure $\wt{\IP}_N$ used to derive the lower bound.

4) Whereas the lower bound (\ref{2.51}) holds for the choice $\wt{L}_0(N) = 0$, i.e.~when one replaces $\wt{\cW}^u_N$ by $\cW^u_N$, the upper bound breaks down in this case. Indeed, $\cW^u_N \supseteq \cI^u \cap B(0,N-1)$, and by the spatial ergodic theorem, one finds that
\begin{equation}\label{2.67}
\mbox{$\IP$-a.s., $\liminf\limits_N \;| \cW^u_N | \, / \,|B(0,N-1)| \ge \IP[0 \in \cI^u] = 1 - e^{-\frac{u}{g(0,0)}} > 0$}
\end{equation}
where $g(x,y)$ stands for the Green function of the simple random walk on $\IZ^d$. One can of course wonder how small $\wt{L}_0(N)$ can be chosen so that (\ref{2.1}) holds. As already mentioned below (\ref{2.5}), explicit rates of decay of the function $\rho(L)$ that appears below (\ref{2.23}), see also (5.6) of \cite{Szni17}, would lead to a more explicit choice of $\wt{L}_0(N)$ in Theorem \ref{theo2.1}. \hfill $\square$
\end{rmk}

\section{The simple random walk}

In this section we now turn to the case of simple random walk on $\IZ^d$, $d \ge 3$, and prove in Theorems \ref{theo3.1} and \ref{theo3.2} the large deviation upper bounds and the lower bound concerning $\wt{\cW}_N$ (see (\ref{0.19})) mentioned in (\ref{0.20}), (\ref{0.22}), (\ref{0.23}). In a heuristic fashion, the simple random walk case discussed in this section corresponds to letting $u$ tend to $0$ in the random interlacement set-up treated in the previous section.

We keep the same notation as introduced below (\ref{0.15}) for the simple random walk $(X_n)_{n \ge 0}$ and its canonical law $P_x$ when starting from $x$ in $\IZ^d$.

The asymptotic upper bounds in the next theorem will come as an application of Theorem \ref{theo2.1} in the previous section and a coupling argument between simple random walk and random interlacements, which can be found in Corollary 7.3 of \cite{Szni15}, and then letting $u$ tend to $0$. We recall that $\ov{u}$ is the critical level for random interlacements, and that $0 < u < \ov{u}$ corresponds to the strongly percolative regime of the vacant set $\cV^u$, see below (\ref{0.6}).

\begin{thm}\label{theo3.1} (upper bound)

There exists $\wt{L}_0(N)$ satisfying (\ref{0.3}) such that for any $\nu > 0$
\begin{equation}\label{3.1}
\limsup\limits_N \; \mbox{\no $\frac{1}{N^{d-2}}$} \;\log P_0[\, | \wt{\cW}_N | \ge \nu N^d] \le - \mbox{\no $\frac{1}{d}$} \; \ov{u} \;{\rm cap}(B_\nu),
\end{equation}
where $\wt{\cW}_N$ is defined in (\ref{0.19}) and $B_\nu$ as below (\ref{0.8}).

In addition, the $\IR^d$-filling $F_N$ of $\frac{1}{n} \;\wt{\cW}_N$, see (\ref{0.21}), is such that for any $\nu > 0$ and $\mu > 0$, one has with similar notation as in (\ref{1.11}):
\begin{equation}\label{3.2}
\begin{array}{l}
\limsup\limits_N \, \frac{1}{N^{d-2}} \,\log P_0[\, | \wt{\cW}_N | \ge \nu N^d, \,\delta(F_N, B_\nu) \ge \mu] 
\\
\le - \frac{1}{d}\; \ov{u} \;\big({\rm cap}(B_\nu) + c_2(\nu,\mu)\big).
\end{array}
\end{equation}
\end{thm}

\begin{proof}
We consider a sequence $\wt{L}_0(N)$ as constructed in Theorem \ref{theo2.1}. For any $u > 0$ (we will actually only consider $u < \ov{u}$), we can find a coupling $P$ of $\cI^u$ under $\IP[ \cdot \, | 0 \in \cI^u]$, and $\cI$ (the set of points in $\IZ^d$ visited by the simple random walk) under $P_0$, such that $\cI \subseteq \cI^{u}$, see the proof of Corollary 7.3 of \cite{Szni15}. Then, we have $P$-a.s., for $N \ge 1$, $\cC_N \supseteq \cC^u_N$ so that $\wt{\cC}^N \supseteq \wt{\cC}^u_N$, and hence 
\begin{equation}\label{3.3}
\mbox{$P$-a.s., for all $N \ge 1, \, \wt{\cW}^u_N \supseteq \wt{\cW}_N$}.
\end{equation}
It then follows that
\begin{equation}\label{3.4}
\begin{array}{l}
\limsup\limits_N   \frac{1}{N^{d-2}} \,\log P_0 [ \,| \wt{\cW}_N| \ge \nu N^d] =  
\\
\limsup\limits_N  \frac{1}{N^{d-2}} \,\log P [ \,| \wt{\cW}_N| \ge \nu N^d]   \le
\\[2ex]
 \limsup\limits_N \frac{1}{N^{d-2}} \,\log P [ \,| \wt{\cW}^u_N| \ge \nu N^d] =
\\
\limsup\limits_N  \frac{1}{N^{d-2}} \,\log \IP [ \,| \wt{\cW}_N^u| \ge \nu N^d\,| \, 0 \in \cI^u] \le
\\[2ex]
\limsup\limits_N  \frac{1}{N^{d-2}} \,\log \IP [ \,| \wt{\cW}^u_N| \ge \nu N^d] \stackrel{{\rm Theor.\,{\scriptsize \ref{2.1}}}}{\le} -  \frac{1}{d} \;(\sqrt{\ov{u}} - \sqrt{u})^2 \,{\rm cap}(B_\nu), 
\end{array}
\end{equation}
where in the inequality at the end of the second line we have used that $\IP[0 \in \cI^u] > 0$. Letting $u \r 0$, we now find (\ref{3.1}).

We now turn to the proof of (\ref{3.2}). We note that by (\ref{3.3}), one has
\begin{equation}\label{3.5}
\mbox{$P$-a.s., for all $N \ge 1$, $F^u_N \supseteq F_N$},
\end{equation}
so that for $N \ge 1$,
\begin{equation}\label{3.6}
\mbox{$P$-a.s., on $\{ | \wt{\cW}_N| \ge \nu N^d\}$, $\nu \le |F_N| \le |F^u_N|$}
\end{equation}
(with $|F_N|$ and $|F^u_N|$ the respective Lebesgue measures of $F_N$ and $F^u_N$).

By (\ref{3.5}), (\ref{3.6}) and the triangle inequality, we find that for $N \ge 1$,
\begin{equation}\label{3.7}
\begin{array}{l}
\mbox{$P$-a.s., on $\{ | \wt{\cW}_N| \ge \nu N^d\}$, $\delta(F_N,B_\nu) \stackrel{\mbox{\f (\ref{3.5})}}{\le}$}
\\
|F^u_N| - |F_N| + \delta(F^u_N, B_\nu) \stackrel{\mbox{\f (\ref{3.6})}}{\le} |F^u_N| - \nu + \delta (F^u_N, B_\nu) \le 2 \delta (F^u_N,B_\nu),
\end{array}
\end{equation}
where we used that $|B_\nu| = \nu$ in the last step. As in (\ref{3.4}), we now find that
\begin{equation}\label{3.8}
\begin{array}{l}
\limsup\limits_N  \mbox{\no $\frac{1}{N^{d-2}}$} \,\log P_0 [ \,| \wt{\cW}_N| \ge \nu N^d, \,\delta(F_N,B_\nu) \ge \mu] = 
\\
\limsup\limits_N  \mbox{\no $\frac{1}{N^{d-2}}$} \,\log P [ \,| \wt{\cW}_N| \ge \nu N^d, \,\delta(F_N,B_\nu) \ge \mu] \stackrel{\mbox{\f (\ref{3.3}}),(\mbox{\f \ref{3.7})}}{\le}
\\
\limsup\limits_N  \mbox{\no $\frac{1}{N^{d-2}}$} \,\log P \big[ \,| \wt{\cW}^u_N| \ge \nu N^d, \,\delta(F^u_N,B_\nu) \ge \mbox{\no $\frac{\mu}{2}$}\big] =
\\
\limsup\limits_N  \mbox{\no $\frac{1}{N^{d-2}}$} \,\log \IP \big[ \,| \wt{\cW}^u_N| \ge \nu N^d, \,\delta(F^u_N,B_\nu) \ge \mbox{\no $\frac{\mu}{2}$}\big] \stackrel{\mbox{\f (\ref{2.2})}}{=} 
\\
- \mbox{\no $\frac{1}{d}$} \,(\sqrt{\ov{u}} - \sqrt{u})^2 \big({\rm cap}(B_\nu) + c_1(\nu,  \mbox{\no $\frac{\mu}{2}$})\big).
\end{array}
\end{equation}
Letting $u$ tend to $0$, we find (\ref{3.2}) with $c_2(\nu,\mu) = c_1(\nu, \frac{\mu}{2})$.
\end{proof}

We now turn to the asymptotic lower bound. As in the case of Theorem \ref{2.4}, the restriction $\nu < \omega_d$ permits to avoid boundary effects (see above (\ref{2.51})). The proof of Theorem \ref{theo3.2} below relies on the change of probability method and the {\it tilted random walk} considered in \cite{Li17}. We recall the notation from (\ref{0.12}) and below (\ref{0.6}).
\begin{thm}\label{theo3.2} (lower bound)

For any $0 < \nu < \omega_d$ and $\wt{L}_0(N)$ satisfying (\ref{0.3}), one has
\begin{equation}\label{3.9}
\liminf\limits_N \; \mbox{\no $\frac{1}{N^{d-2}}$}\,\log P_0 \big[ \,| \wt{\cW}_N| \ge \nu N^d] \ge - \mbox{\no $\frac{1}{d}$} \;u_{**} \,{\rm cap}(B_\nu)
\end{equation}
with $B_\nu$ as below (\ref{0.8}). 
\end{thm}

\begin{proof}
We consider $0 < \nu < \omega_d$, and $R_\nu \in (0,1)$ as in (\ref{2.52}). We further consider $\delta \in (0,1)$ such that $R_\nu + 2 \delta < 1$ (as in (\ref{2.53})), and for $N \ge 1$, define $\Gamma_N$ as in (\ref{2.59}). We then consider the measure $\wt{P}_N$ from (2.36) of \cite{Li17} that governs the so-called ``tilted random walk'' in that work. We refer to \cite{Li17} for the precise construction of $\wt{P}_N$, but in essence one considers a compactly supported non-negative approximation $\wt{h}$ of the equilibrium potential of $B_\nu$ (on $\IR^d$), and sets $h_N(x) = \wt{h}(\frac{x}{N})$. Under $\wt{P}_N$ (see Lemma 2.5 of \cite{Li17}) the corresponding walk starts at the origin, behaves as a Markov chain with generator $\wt{L}_N \,g(x)=\frac{1}{2d} \sum_{|y - x|_1 = 1} \frac{h_N(y)}{h_N(x)} \,(g(y) - g(x))$ up to a deterministic time $T_N=(u_{**} + \epsilon) \,\sum_{y \in \IZ^d} h^2_N(y) $, and then moves as a simple random walk on $\IZ^d$ with unit jump rate.  By Proposition 5.8 of \cite{Li17} we know that
\begin{equation}\label{3.10}
\wt{P}_N [ \cA_N] \underset{N}{\longrightarrow} 1, \;\; \mbox{if} \;\; \cA_N = \mbox{\normalsize $\bigcap\limits_{x \in \Gamma_N}$} \big\{ x \stackrel{\cV}{\mbox{\Large $\longleftrightarrow$}} \hspace{-3.5ex} /\quad S_{N-1}\big\}.
\end{equation}
(i.e. $\cA_N$ is the event stating that no site of $\Gamma_N$ is linked by a path in $\cV = \IZ^d \backslash \cI$ to a site with $| \cdot|_\infty$-norm $N-1$). For large $N$, on $\cA_N$ one has $\cC_N \subseteq B(0,N) \backslash (B(0,R_\nu + \frac{\delta}{2})\,N) \cap \IZ^d)$ so that
\begin{equation*}
\wt{\cW}_N \supseteq B\big(0, \big(R_\nu  + \mbox{\no $\frac{\delta}{4}$}\big)\,N\big) \cap \IZ^d.
\end{equation*}
It thus follows that
\begin{equation}\label{3.11}
\mbox{for large $N$, on $\cA_N$, $|\wt{\cW}_N| \ge \nu N^d$}.
\end{equation}
By the change of probability method, see (\ref{1.6}), it now follows that
\begin{equation}\label{3.12}
\liminf\limits_N \; \mbox{\no $\frac{1}{N^{d-2}}$}\,\log P_0 \big[ \,| \wt{\cW}_N| \ge \nu N^d] \ge - \limsup\limits_N \, \mbox{\no $\frac{1}{N^{d-2}}$} \;H(\wt{P}_N\,| \,\ov{P}_0)
\end{equation}
(where $\ov{P}_0$ governs the law of the continuous time simple random walk on $\IZ^d$ with unit jump rate and starting at the origin).

Proceeding as in (6.3) of \cite{Li17}, we then obtain by taking successive limits in the parameters entering smoothing, cut-off, and threshold, in the construction of the tilted walk (similar to (\ref{2.64})) that
\begin{equation}\label{3.13}
\liminf\limits_N \; \mbox{\no $\frac{1}{N^{d-2}}$}\,\log P_0 \big[ \,| \wt{\cW}_N| \ge \nu N^d] \ge - \mbox{\no $\frac{1}{d}$} \; u_{**}\, {\rm cap}(B_\nu).
\end{equation}
This proves (\ref{3.9}).
\end{proof}

\begin{rmk}\label{rem3.3} \rm  ~

1) As in Remark \ref{rem2.5} 2), let us point out that under $\wt{P}_N$ with overwhelming probability as $N$ tends to infinity, there are no macroscopic components in $\cV \cap B_2 (0, R_\nu N)$.

2) If the equalities $\ov{u} = u_* = u_{**}$ hold, then choosing $\wt{L}_0(N)$ as in Theorem \ref{theo3.1}, for $0 < \nu < \omega_d$ the upper and lower bounds in (\ref{3.1}), (\ref{3.9}) are matching so that with $B_\nu$ as below (\ref{0.8}), one has
\begin{equation}\label{3.14}
\lim\limits_N \; \mbox{\no $\frac{1}{N^{d-2}}$}\,\log P_0 \big[ \,| \wt{\cW}_N| \ge \nu N^d] = - \mbox{\no $\frac{1}{d}$} \; u_*\, {\rm cap}(B_\nu), \;\mbox{for $0 < \nu < \omega_d$}.
\end{equation}
In addition, by (\ref{3.2}), conditionally on $\{ \,| \wt{\cW}_N| \ge \nu N^d\}$, the $\IR^d$-filling $F_N$ of $\frac{1}{N}\wt{\cW}_N$ is close in $L^1$-distance to a translate of $B_\nu$:
\begin{equation}\label{3.15}
\lim\limits_N \; E_0[\delta (F_N,B_\nu) \,\big| \, |\wt{\cW}_N| \ge \nu N^d] = 0, \;\mbox{for $0 < \nu < \omega_d$}.
\end{equation}

However, what happens for large $N$  under $P_0[\,\, \cdot \, \, \big| \, |\wt{\cW}_N| \ge \nu N^d]$ inside the ``nearly spherical hole'' $\wt{\cW}_N$ (left in $B(0,N)$ by $\wt{\cC}_N$) is unclear. In particular are large connected components of $\cV$ present or not? The picture may be quite different from what takes place under $\wt{P}_N$. See 1) above, and below (\ref{2.66}) for a similar remark in the case of random interlacements.

3) In Theorem \ref{theo3.1} the proof shows that $\wt{L}_0(N)$ can be chosen as in Theorem \ref{theo2.1}. It remains an open issue to find out how small $\wt{L}_0(N)$ can be chosen so that (\ref{3.1}) and (\ref{3.2}) hold. \hfill $\square$
\end{rmk}

\section{The Gaussian free field}

We now turn to the case where the level set $E^{\ge \alpha} = \{\varphi \ge \alpha\}$ of the canonical Gaussian free field on $\IZ^d$, $d \ge 3$, in the strongly percolative regime $\alpha < \ov{h}$, replaces the vacant set $\cV^u$ random interlacements on $\IZ^d$, with $0 < u < \ov{u}$, in Section 2, or the vacant set $\cV$ of the simple random walk on $\IZ^d$, in Section 3. Our main results in this section pertain to large deviation bounds on the volume and shape of the complement $\wt{W}^{\ge \alpha}_N$ in $B(0,N)$ of the suitably thickened component of $S_N$ in $E^{\ge \alpha}$. They appear in Theorems \ref{theo4.1} and \ref{theo4.2} of this section. We keep the same notation as below (\ref{0.25}).

The next theorem is the counterpart to Theorem \ref{theo2.1} in the present context.

\begin{thm}\label{theo4.1} (upper bound)

There exists $\wt{L}_0(N)$ satisfying (\ref{0.3}) such that for any $\alpha < \ov{h}$ and $\nu > 0$
\begin{equation}\label{4.1}
\limsup\limits_N  \mbox{\no $\frac{1}{N^{d-2}}$} \,\log \IP^G [ \,| \wt{\cW}^{\,\ge \alpha}_N| \ge \nu N^d] \le - \mbox{\no $\frac{1}{2d}$} \;(\ov{h} - \alpha)^2 \,{\rm cap}(B_\nu)
\end{equation}
(with $B_\nu$ as below (\ref{0.8})).

In addition, for the $\IR^d$-filling $F^{\ge \alpha}_N$ of $\frac{1}{N} \;\wt{\cW}^{\, \ge \alpha}_N$, see (\ref{0.32}), one has for any $\nu > 0$ and $\mu > 0$ (with similar notation as in (\ref{1.11}))
\begin{equation}\label{4.2}
\begin{array}{l}
\limsup\limits_N  \mbox{\no $\frac{1}{N^{d-2}}$} \,\log \IP^G [ \,| \wt{\cW}^{\,\ge \alpha}_N| \ge \nu N^d, \,\delta(F^{\ge \alpha}_N, B_\nu) \ge \mu] 
\\
\le - \mbox{\no $\frac{1}{2d}$} \;(\ov{h} - \alpha)^2 \,\big({\rm cap}(B_\nu) + c_1(\nu,\mu)\big).
\end{array}
\end{equation}
\end{thm}

\begin{proof}
Once again, the proof heavily relies on a coarse graining procedure to construct a porous interface as in Section 4 of \cite{NitzSzni}, and on the capacity lower bound (1.4) proved in Corollary 3.4 of \cite{NitzSzni}, see also Section 3 of \cite{Nitz}. The general line is similar to the proof of Theorem \ref{theo2.1}, and we will highlight where different arguments are involved.

We first specify $\wt{L}_0(N)$. To this end, for each rationals $\delta < \gamma < \ov{h}$ and integer $K \ge 100$, we choose a sequence $\gamma_N$ (depending on $\delta < \gamma$ and $K$), which tends to $0$ as $N \r \infty$, and  satisfies the assumptions corresponding to (4.18) of \cite{NitzSzni}, but the function $\rho(L)$ above (4.15) of \cite{NitzSzni} is now replaced by the function in (5.18) of \cite{Szni15} (which depends on $\gamma, \delta$ and $K$), see also Section 3 of \cite{Nitz}. We can then define $L_0$ and $\wh{L}_0$ as in (\ref{2.3}) and choose $\wt{L}_0(N)$ to be any non-negative integer valued sequence such that
\begin{equation}\label{4.3}
\left\{ 
\begin{array}{rll}
{\rm i)} &\!\!\! \wt{L}_0(N) = o(N),
\\[1ex]
{\rm ii)} &\!\!\! \wt{L}_0(N) \, / \, \wh{L}_0(N) \underset{N}{\longrightarrow} \infty& \!\!\!\mbox{for any choice of rationals $\delta < \gamma < \ov{h}$}
\\[-1ex]
&&\!\!\! \mbox{and integer $K \ge 100$}
\end{array}\right.
\end{equation}
(such a sequence can for instance be constructed by a diagonal procedure, see also the end of Remark \ref{rem4.3} 3)). We then introduce the lattices $\IL_0$ and $\wh{\IL}_0$ as in (\ref{2.6}) and keep the same notation $B_z,D_z$, and $U_z$ for $z \in \IZ^d$, as in (\ref{2.7}).

We now choose $\delta < \gamma$ rationals in $(\alpha, \ov{h})$ and $K \ge 100$ (and the corresponding sequence $\gamma_N$ mentioned above). There is a notion of $B_z$ being a $\psi$-good box, see (5.7), (5.8) of \cite{Szni15}. The details are not important here. It corresponds to a decomposition for each box $B = B_z$ of the Gaussian free field $\varphi$ into $\varphi = h_B + \psi_B$, where $h_B$ is an harmonic average of $\varphi$ inside $U_z$ (and equals $\varphi$ outside $U_z$) and $\psi_B$ is a local field. As mentioned above (\ref{2.8}), $h_B$ aims at tracking an ``undertow'' in the field, and $\psi_B$ contains a spatially faster decorrelating information. In essence for a $\psi$-good box $B_z$ at level $\delta < \gamma$, $B_z \cap \{\psi_{B_z} \ge \gamma\}$ contains a component of $| \cdot |_\infty$-diameter at least $L_0/10$, and for any two neighboring boxes $B_z$ and $B_{z'}$, any two components of $B_z \cap \{\psi_{B_z} \ge \gamma\}$ and $B_{z'} \cap \{\psi_{B_{z'}} \ge \gamma\}$ with diameter at least $L_0/10$ are connected in $D_z \cap \{\psi_{B_z} \ge \delta\}$.

There is also a notion of a box $B_z$ being $h$-good at level $a > 0$, namely that $\inf_{D_z} h_{B_z} > -a$ (for us $a = \delta - \alpha > 0$ will be the natural choice, and eventually we will let $\delta$ and $\gamma$ tend to $\ov{h}$). With these notions in mind, we define similarly as in (\ref{2.8})
\begin{equation}\label{4.4}
\begin{array}{ll}
\cU^1 =  &\!\!\! \mbox{the union of all $L_0$-boxes that are either contained in}
\\
&\!\!\! \mbox{$B(0,2N)^c$ or linked to an $L_0$-box contained in $B(0,2N)^c$} 
\\
&\!\!\! \mbox{by a path of $L_0$-boxes $B_{z_i}, 0 \le i \le n$, which are all, except} 
\\
&\!\!\! \mbox{may be for the last one, $\psi$-good at level $\delta, \gamma$ and $h$-good}
\\
&\!\!\!  \mbox{at level $a = \delta - \alpha$}. 
\end{array}
\end{equation}
Parallel to (\ref{2.9}) and (\ref{2.10}), we introduce the ``hole events''
\begin{equation}\label{4.5}
\begin{split}
\cH^{\ge \alpha}_N &= \{| \wt{\cW}^{\,\ge \alpha}_N| \ge N^d\} \supseteq \cH^{\ge\alpha,\mu}_N  
\\
& = \{| \wt{\cW}^{\,\ge \alpha}_N| \ge \nu N^d, \, \delta(F^{\ge \alpha}_N ,B_\nu) \ge \mu\}.
\end{split}
\end{equation}
We apply the same coarse graining procedure with the help of the function $\wh{\sigma}(x) = |\cU^1 \cap B(x,\wh{L}_0)| \, / \, |B(x,\wh{L}_0)|$, $x \in \IZ^d$. We introduce a bad event $\cB_N$ as in (\ref{2.23}), with our now current choice of the function $\rho(\cdot)$, and $\psi$-bad at levels $\delta,\gamma$ in place of bad$(\alpha,\beta,\gamma)$. The bad event $\cB_N$ satisfies a super-exponential bound as in (\ref{2.25}), see also Proposition 5.4 of \cite{Szni15}, or (3.17) of \cite{Nitz}. Similar to (\ref{2.26}), we then introduce the effective events
\begin{equation}\label{4.6}
\wt{\cH}^{\,\ge \alpha}_N = \cH^{\ge \alpha}_N \backslash \cB_N \supseteq  \wt{\cH}^{\,\ge \alpha,\mu}_N = \cH^{\ge \alpha,\mu}_N  \backslash \cB_N.
\end{equation}
As in (\ref{2.16}), (\ref{2.22}), (\ref{2.28}), for large $N$ we construct on $\cH^{\ge \alpha}_N$ a random variable $\kappa_N = (\wh{\cS}_N, \wt{\cS}_N, (\wt{\pi}_x,\wt{\cC}_x)_{x \in \wt{\cS}_N})$ with range $\cK_N$, a set such that $| \cK_N| = \exp\{o(N^{d-2})\}$, so that corresponding to (\ref{2.27}), see also below (3.16) of \cite{Nitz}, we have with $\ov{K} = 2K+3$
\begin{equation}\label{4.7}
\begin{array}{l}
\mbox{for large $N$, on $\cH^{\ge \alpha}_N$, for each $x \in \wt{\cS}_N$, $\wt{\cC}_x$ is a collection of}
\\
\mbox{$L_0$-boxes, which intersect $B(x,\wh{L}_0)$, have $\wt{\pi}_x$-projection at}
\\
\mbox{mutual distance at least $\ov{K} L_0$, with cardinality $[(\frac{c'}{K} \;\frac{\wh{L}_0}{L_0})^{d-1}]$,}
\\
\mbox{and such that for each $z \in \wt{C}_x$, $B_z$ is $\psi$-good at level $\gamma,\delta$ and}
\\
\mbox{$h$-bad at level $a = \delta - \alpha$}.
\end{array}
\end{equation}
One then attaches to each $\kappa \in \cK_N$ a ``segmentation'' represented by $S$ (or $U_0$), and a porous interface corresponding to $\Sigma \subseteq \IR^d$, the solid filling of $\frac{1}{N} \,C$ (with $C \subseteq \IZ^d)$, as in (\ref{2.31}).

Then, with Corollary 4.4 of \cite{Szni15}, we have a function $\alpha(K)>1$ with $\lim_{K \r \infty} \alpha(K) = 1$, so that setting $\cH_{N,\kappa} = \cH^{\ge \alpha}_{N} \cap \{\kappa_N = \kappa\}$, for $\kappa \in \cK_N$, we have
\begin{equation}\label{4.8}
\limsup\limits_N \;\sup\limits_{\kappa \in \cK_N} \Big\{\log \IP^G[\cH_{N,\kappa}] + \mbox{\no $\frac{1}{2}$} \;\Big(a - \mbox{\no $\frac{c}{K}$} \;\sqrt{\mbox{\no $\frac{|\cC |}{{\rm cap}_{\IZ^d}(C)}$}}\Big)_+^2 \mbox{\no $\frac{{\rm cap}_{\IZ^d}(C)}{\alpha(K)}$} \Big\} \le 0,
\end{equation}
where for $\kappa \in \cK_N$, $C$ is as in (\ref{2.31}) and we used the notation
\begin{equation}\label{4.9}
\cC = \bigcup\limits_{x \in \wt{S}} \; \wt{\cC}_x,
\end{equation}
so that for large $N$, by (\ref{2.3})
\begin{equation}\label{4.10}
\begin{array}{l}
|\cC | \le  \mbox{\normalsize $\sum\limits_{x \in \wt{S}} |\wt{\cC}_x| \le c \,\big(\frac{N}{\wh{L}_0}\big)^d \big( \frac{c'}{K} \; \frac{\wh{L}_0}{L_0}\big)^{d-1}$} 
\\
\le \mbox{\normalsize $\frac{c}{K^{d-1}} \;\gamma^{-\frac{1}{2}}_N \big(\frac{N}{L_0}\big)^{d-1} \le \frac{c}{K^{d-1}} \;\; \gamma^{\frac{1}{2}}_N \;\;\frac{N^{d-2}}{\log N}$}.
\end{array}
\end{equation}
As we will now explain, the term following $a$ inside the square in (\ref{4.8}) becomes negligible as we successively let $N$ tend to infinity and then $K$ tend to infinity. We first note that by Proposition A.1 of \cite{NitzSzni} one has (with $C$ and $\Sigma$ as in (\ref{2.31}))
\begin{equation}\label{4.11}
\liminf\limits_{K \r \infty} \; \;\liminf\limits_{N \r \infty} \;\; \inf\limits_{\kappa \in \cK_N} \; \mbox{\no $\frac{1}{N^{d-2}}$} \; {\rm cap}_{\IZ^d} (C) / {\rm cap}(\Sigma) \ge \mbox{\no $\frac{1}{d}$}.
\end{equation}
Moreover, with $A_\kappa$ as in (\ref{2.33}), by the capacity lower bound (\ref{1.4}), with $\ve = 10 \, \wh{L}_0 / N,\ell_*$ the smallest non-negative integer such that $2^{-\ell_*} \le \wt{L}_0 / (10N)$, and $A = A_\kappa$, we find that
\begin{equation}\label{4.12}
\liminf\limits_N \;\;\inf\limits_{\kappa \in \cK_N} \;{\rm cap}(\Sigma) \, / \, {\rm cap}(A_\kappa) \ge 1,
\end{equation}
and as in (\ref{2.34}),
\begin{equation}\label{4.13}
\begin{array}{l}
\mbox{for large $N$, for any $\kappa \in \cK_N$, on $\cH_{N,\kappa}$},
\\
\mbox{one has $F_N^{\ge \alpha} \subseteq \mathring{A}_\kappa$ and $\nu \le |F_N^{\ge \alpha}| \le |\mathring{A}_\kappa|$}.
\end{array}
\end{equation}
By the Polya-Szeg\"o Inequality, see below (\ref{1.8}), we then have ${\rm cap}(\mathring{A}_\kappa) \ge {\rm cap}(B_\nu)$, and coming back to (\ref{4.12}), we find that
\begin{equation}\label{4.14}
\underset{N}{\underline{\lim}} \;\;\inf\limits_{\kappa \in \cK_N} \;{\rm cap}(\Sigma) \, / \, {\rm cap}(B_\nu) \ge \underset{N}{\underline{\lim}} \; \inf\limits_{\kappa \in \cK_N} \;{\rm cap}(A_\kappa) \, / \, {\rm cap}(B_\nu) \ge 1.
\end{equation}
As a result, combining (\ref{4.10}), (\ref{4.11}), (\ref{4.14}), we find that
\begin{equation}\label{4.15}
\limsup\limits_K \;  \limsup\limits_N \;  \sup\limits_{\kappa \in \cK_N} \;  |\cC| \, / \, {\rm cap}_{\IZ^d}(C) = 0.
\end{equation}
If we now proceed as in (\ref{2.46}), (\ref{2.47}), taking into account the super-exponential bound satisfied by the bad event $\cB_N$, we obtain
\begin{equation}\label{4.16}
\begin{array}{l}
\limsup\limits_N \;  \mbox{\no $\frac{1}{N^{d-2}}$} \;\log \IP^G [\cH^{\ge \alpha}_N] \;\le
\\[1ex]
- \underset{K \r \infty}{\underline{\lim}} \;\; \underset{N \r \infty}{\underline{\lim}} \; \inf\limits_{\kappa \in \cK_N} \; \mbox{\no $\frac{1}{N^{d-2}}$} \;\log \IP^G [\cH_{N,\kappa}] \underset{\mbox{\f (\ref{4.12}), (\ref{4.15})}}{\stackrel{\mbox{\f (\ref{4.8}),(\ref{4.11})}}{\le}}
\\[3ex]
- \underset{K \r \infty}{\underline{\lim}} \; \underset{N \r \infty}{\underline{\lim}}\; \inf\limits_{\kappa \in \cK_N}  \; \mbox{\no $\frac{a^2}{\alpha(K)}$}\; \mbox{\no $\frac{1}{d}$} \;{\rm cap}(A_\kappa)  \underset{\alpha(K) \r 1}{\stackrel{\mbox{\f (\ref{4.14})}}{\le}}  -\mbox{\no $\frac{a^2}{2d}$} \;{\rm cap}(B_\nu).
\end{array}
\end{equation}
Recall that $a = \delta - \alpha$. Letting $\delta$ converge (along rationals) to $\ov{h}$, we find (\ref{4.1}). As for (\ref{4.2}), we introduce as in (\ref{2.36}) for large $N$
\begin{equation}\label{4.17}
\cK^\mu_N = \{\kappa \in \cK_N; \cH_{N,\kappa} \cap \wt{\cH}^{\,\ge \alpha,\mu}_N \not= \phi\}
\end{equation}
so that for large $N$, as in (\ref{2.37}),
\begin{equation}\label{4.18}
\mbox{for any $\kappa \in \cK^\mu_N, | \mathring{A}_\kappa| \ge \nu \; \mbox{and} \; 2 (| \mathring{A}_\kappa| - \nu) + |\mathring{A}_\kappa| \,\lambda_{\mathring{A}_\kappa} \ge \mu $}.
\end{equation}
We then have, proceeding as in (\ref{4.16}),
\begin{equation}\label{4.19}
\begin{array}{l}
\limsup\limits_N \; \mbox{\no $\frac{1}{N^{d-2}}$} \; \log \IP^G [\cH^{\,\ge \alpha, \mu}_N]  \le 
\\
- \underset{K \r \infty}{\underline{\lim}} \;\; \underset{N \r \infty}{\underline{\lim}} \;\inf\limits_{\kappa \in \cK^\mu_N} \;\; \mbox{\no $\frac{a^2}{\alpha(K)}$} \; \mbox{\no $\frac{1}{d}$} \;{\rm cap}(A_\kappa)  
\\[2ex]
 \underset{\alpha(K) \r 1}{\stackrel{\mbox{\f (\ref{4.18}),(\ref{2.39})}}{\le}}  - \;\mbox{\no $\frac{a^2}{2d}$} \; \big({\rm cap}(B_\nu) + c_1(\nu,\mu)\big).
\end{array}
\end{equation}
Letting $\delta$ converge along rationals to $\ov{h}$ now yields (\ref{4.2}). This concludes the proof of Theorem \ref{theo4.1}.
\end{proof}

We now come to the asymptotic lower bound. As explained above Theorem \ref{theo2.4}, the condition $\nu < \omega_d$ ensures that $B_\nu$ ist contained in $\mathring{B}(0,1)$ and permits us to avoid boundary effects. We refer to above (\ref{0.26}) and to (\ref{0.29}) for notation.

\begin{thm}\label{theo4.2} (lower bound)

For any $\alpha < h_{**}$, $0 < \nu < \omega_d$, and $\wt{L}_0(N)$ satisfying (\ref{0.3}), we have
\begin{equation}\label{4.20}
\liminf\limits_N \; \mbox{\no $\frac{1}{N^{d-2}}$} \; \log \IP^G [\,| \wt{\cW}^{\, \ge \alpha}_N| \ge \nu N^d] \ge - \mbox{\no $\frac{1}{2d}$} \;(h_{**} - \alpha)^2 \,{\rm cap}(B_\nu).
\end{equation}
\end{thm}

\begin{proof}
We keep the notation from (\ref{2.52}), (\ref{2.53}), so that $\nu = \omega_d \,R^d_\nu$, and we have $0 < \delta < 1$ and $r > 1$ satisfying
\begin{equation}\label{4.21}
R_\nu + 2 \delta < 1 < r,
\end{equation}
as well as $0 < \eta < \delta$ and the regularized function $h^\eta = h * \phi^\eta$, as in (\ref{2.54}), (\ref{2.55}). We then define for $\ve > 0$ and $N \ge 1$ (recall that $\alpha < h_{**}$)
\begin{equation}\label{4.22}
f_N(x) = - h^\eta \big(\mbox{\no $\frac{x}{N}$}\big) \; (h_{**} - \alpha + \ve), \;\mbox{for $x \in \IZ^d$},
\end{equation}
and
\begin{equation}\label{4.23}
\wt{\IP}^G_N = \exp\Big\{\cE(f_N,\varphi) - \mbox{\no $\frac{1}{2}$} \;\cE(f_N,f_N)\Big\} \,\IP^G,
\end{equation}
where $\cE(\cdot,\cdot)$ stands for the Dirichlet form
\begin{equation}\label{4.24}
\cE(f,g) = \mbox{\no $\frac{1}{2}$} \; \dis\sum\limits_{|x - y|_1 = 1} \; \mbox{\no $\frac{1}{2d}$} \; \big(f(y) - f(x)\big) \,\big(g(y) - g(x)\big),
\end{equation}
for $f,g$ functions on $\IZ^d$ such that the above series is absolutely convergent. By the Cameron-Martin formula (see also (2.3), (2.4) of \cite{Szni15}), we know that $\wt{\IP}^G_N$ is a probability measure and that
\begin{equation}\label{4.25}
\mbox{$\varphi$ under $\wt{\IP}^G_N$ has the same law as $\varphi + f_N$ under $\IP^G$}.
\end{equation}
Then, observe that
\begin{equation}\label{4.26}
f_N(x) = - (h_{**} - \alpha + \ve) \; \mbox{for $x$ in $B_2(0,(R_\nu + \delta)\,N) \cap \IZ^d$},
\end{equation}
and defining $\Gamma_N$ as in (\ref{2.59}) to be the exterior boundary of the discrete blow-up of $B_2(0,R_\nu + \frac{\delta}{2})$, we see that for large $N$ and any $x \in \Gamma_N$,  we have (with $\partial_{\rm int}$ denoting the inner boundary):
\begin{equation}\label{4.27}
\begin{array}{l}
\wt{\IP}^G_N \big[ x \stackrel{\varphi \ge \alpha}{\longleftrightarrow} \partial_{\rm int} \,B\big(x, \mbox{\no $\frac{\delta}{4d}$} \,N\big)\big] \stackrel{\mbox{\f (\ref{4.25})}}{=} 
\\
\IP^G \big[ x \stackrel{\varphi \ge \alpha - f_N}{\longleftrightarrow} \partial_{\rm int} \,B\big(x, \mbox{\no $\frac{\delta}{4d}$} \,N\big)\big] \stackrel{\mbox{\f (\ref{4.26})}}{=}
\\[1ex]
\IP^G \big[ x \stackrel{\varphi \ge h_{**} + \ve}{\longleftrightarrow} \partial_{\rm int} \,B\big(x, \mbox{\no $\frac{\delta}{4d}$} \,N\big)\big] \stackrel{{\rm trans.\,inv.}}{=} 
\\
\IP^G \big[ 0 \stackrel{\varphi \ge h_{**} + \ve}{\longleftrightarrow} \partial_{\rm int} \,B\big(0, \mbox{\no $\frac{\delta}{4d}$} \,N\big)\big] .
\end{array}
\end{equation}
This last quantity has a stretched exponential decay in $N$ (actually an exponential decay when $d \ge 4$, with a logarithmic correction when $d=3$, see Theorem 2.1 of \cite{PopoRath15}). It thus follows that
\begin{equation}\label{4.28}
\wt{\IP}^G_N [\cA^{\ge \alpha}_N] \underset{N}{\longrightarrow} 1, \; \mbox{if} \; \cA_N^{\ge \alpha} = \mbox{\no $\bigcap\limits_{x \in \Gamma_N}$} \big\{ x \stackrel{\varphi \ge \alpha}{\mbox{\Large $\longleftrightarrow$}} \hspace{-3.5ex} /\quad S_N\big\}.
\end{equation}
By the change of probability method, see (\ref{1.6}), and the fact that for large $N$ one has $\cA^{\ge \alpha}_N \subseteq \{\,| \wt{\cW}^{\ge \alpha}_N| \ge \nu N^d\}$ (as in (\ref{2.61})), one finds that
\begin{equation}\label{4.29}
\begin{split}
\liminf\limits_N \; \mbox{\no $\frac{1}{N^{d-2}}$} \;\log \IP^G [\,| \wt{\cW}^{\ge \alpha}_N| \ge \nu N^d] & \ge - \limsup\limits_N \; \mbox{\no $\frac{1}{N^{d-2}}$} \;H(\wt{\IP}^G_N \,| \, \IP^G)
\\
& = - \limsup\limits_N \; \mbox{\no $\frac{1}{N^{d-2}}$} \; \; \mbox{\no $\frac{1}{2}$} \;\cE(f_N,f_N)
\end{split}
\end{equation}
(see (2.7) of \cite{Szni15} for the last equality).

By a similar calculation as below (2.10) of \cite{Szni15}, we find that
\begin{equation}\label{4.30}
\begin{array}{l}
\liminf\limits_N \; \frac{1}{N^{d-2}} \;\log \IP^G [\,| \wt{\cW}^{\ge \alpha}_N| \ge \nu N^d] 
\\[1.5ex]
\ge  - \frac{1}{2d} \;(h_{**} - \alpha + \ve)^2 \,{\rm cap} \big(B_2(0,R_\nu + \delta)\big),
\end{array}
\end{equation}
and letting $\delta$ and $\ve$ tend to zero, we obtain (\ref{4.20}).
\end{proof}

Similarly to Remark \ref{rem2.5}, we now have

\begin{rmk}\label{rem4.3} \rm  ~

1) For large $N$ under $\wt{\IP}^G_N$ the super level-set $\{\varphi \ge \alpha\}$ inside $B_2(0,R_\nu N) \cap \IZ^d$ is in a strongly non-percolative regime as shown by a similar calculation as in (\ref{4.27}). In particular, with overwhelming 
$\wt{\IP}^G_N$-probability as $N \r \infty$, there are no macroscopic components in $\{\varphi \ge \alpha\} \cap B_2(0,R_\nu N)$. Once again this is qualitatively different from what would happen in the case of Bernoulli percolation, see Remark \ref{rem2.5} 2).

2) If the equalities $\ov{h} = h_* = h_{**}$ hold, then choosing $\wt{L}_0(N)$ as in Theorem \ref{theo4.1} the upper and lower bounds in (\ref{4.1}) and (\ref{4.20}) are matching when $\alpha < h_*$ and $0 < \nu < \omega_d$, so that (with $B_\nu$ as below (\ref{0.8}))
\begin{equation}\label{4.31}
\lim\limits_N \; \mbox{\no $\frac{1}{N^{d-2}}$} \;\log \IP^G  [\,| \wt{\cW}^{\ge \alpha}_N| \ge \nu N^d] = - \mbox{\no $\frac{1}{2d}$} \;(h_{*} - \alpha)^2 \,{\rm cap} (B_\nu).
\end{equation}
In addition by (\ref{4.2}) and (\ref{4.20}), conditionally on $| \wt{\cW}^{\ge \alpha}_N| \ge \nu N^d$, the $\IR^d$-filling $F^{\ge \alpha}_N$ of $\frac{1}{N} \; \wt{\cW}^{\ge \alpha}_N$ is close to a translate of $B_\nu$ in the sense that
\begin{equation}\label{4.32}
\lim\limits_N \;\IE^G [\delta(F^{\ge \alpha}_N, B_\nu) \, \big| \, | \wt{\cW}^{\ge \alpha}_N| \ge \nu N^d]  = 0
\end{equation}
(see (\ref{1.11}) for notation).

Again, what happens for large $N$  under $\IP^G[\,\, \cdot \, \, \big| \, |\wt{\cW}^u_N| \ge \nu N^d]$ inside the ``nearly spherical hole'' $\wt{\cW}^{\ge \alpha}_N$ (left in $B(0,N)$ by $\wt{\cC}^{\ge \alpha}_N$) is unclear, and possibly quite different from what takes place under the measure $\wt{\IP}^G_N$ used for the lower bound. See 1) above, and below (\ref{2.66}) and (\ref{3.15}) for similar remarks in the case of random interlacements and of the simple random walk.

3) Let us point out that the upper bound (\ref{4.1}) does not hold for the choice $\wt{L}_0(N) = 0$ (because $\IP^G$-a.s., $\underline{\lim}_N$ $| \wt{\cW}^{\ge \alpha}_N|  \,/ \,|B(0,N)| \ge \IP^G [\varphi_0 < \alpha] > 0$, by the application of the ergodic theorem and the inclusion $\{\varphi < \alpha\} \cap B(0,N-1) \subseteq B(0,N) \backslash \cC_N^{\ge N} = \cW_N^{\ge \alpha}$). One can naturally wonder how small $\wt{L}_0(N)$ can be chosen so that (\ref{4.1}) holds. Incidentally, explicit rates of decay on the function $\rho(L)$ from (5.18) of \cite{Szni15} would lead to a more explicit choice of $\wt{L}_0(N)$ in Theorem \ref{theo4.1}.  

\hfill $\square$
\end{rmk}

\end{document}